\documentclass[10pt]{amsart}
\usepackage{amssymb, enumitem}
\usepackage[all]{xy}
\usepackage{hyperref, aliascnt}
\usepackage{mathtools}
\usepackage[english]{babel}
\usepackage{mathabx}
\usepackage[T1]{fontenc}
\def\today{\number\day\space\ifcase\month\or   January\or February\or
   March\or April\or May\or June\or   July\or August\or September\or
   October\or November\or December\fi\   \number\year}

\theoremstyle{definition}
\newtheorem{lma}{Lemma}[section]

\newaliascnt{thmCt}{lma}
\newtheorem{thm}[thmCt]{Theorem}
\aliascntresetthe{thmCt}

\newaliascnt{corCt}{lma}
\newtheorem{cor}[corCt]{Corollary}
\aliascntresetthe{corCt}

\newaliascnt{propCt}{lma}
\newtheorem{prop}[propCt]{Proposition}
\aliascntresetthe{propCt}

\newtheorem*{thm*}{Theorem}
\newtheorem*{cor*}{Corollary}
\newtheorem*{prop*}{Proposition}

\newcounter{theoremintro}
\newtheorem{thmintro}[theoremintro]{Theorem}
\newtheorem{corintro}[theoremintro]{Corollary}

\newaliascnt{pgrCt}{lma}

\aliascntresetthe{pgrCt}

\newaliascnt{dfCt}{lma}
\newtheorem{df}[dfCt]{Definition}
\aliascntresetthe{dfCt}

\newaliascnt{remCt}{lma}
\newtheorem{rem}[remCt]{Remark}
\aliascntresetthe{remCt}

\newaliascnt{remsCt}{lma}

\aliascntresetthe{remsCt}

\newaliascnt{egCt}{lma}
\newtheorem{eg}[egCt]{Example}
\aliascntresetthe{egCt}

\newaliascnt{egsCt}{lma}

\aliascntresetthe{egsCt}

\newaliascnt{qstCt}{lma}
\newtheorem{qst}[qstCt]{Question}
\aliascntresetthe{qstCt}

\newaliascnt{pbmCt}{lma}

\aliascntresetthe{pbmCt}

\newaliascnt{notaCt}{lma}
\newtheorem{nota}[notaCt]{Notation}
\aliascntresetthe{notaCt}

\newaliascnt{cnjCt}{lma}

\aliascntresetthe{cnjCt}

\newcommand{\beq}{\begin{equation}}
\newcommand{\eeq}{\end{equation}}
\newcommand{\beqa}{\begin{eqnarray*}}
\newcommand{\eeqa}{\end{eqnarray*}}
\newcommand{\bal}{\begin{align*}}
\newcommand{\eal}{\end{align*}}
\newcommand{\bi}{\begin{itemize}}
\newcommand{\ei}{\end{itemize}}
\newcommand{\be}{\begin{enumerate}}
\newcommand{\ee}{\end{enumerate}}

\newcommand{\dt}{\delta}
\newcommand{\ep}{\varepsilon}
\newcommand{\zt}{\zeta}

\newcommand{\Z}{{\mathbb{Z}}}
\newcommand{\R}{{\mathbb{R}}}
\newcommand{\C}{{\mathbb{C}}}
\newcommand{\N}{{\mathbb{N}}}

\newcommand{\K}{{\mathcal{K}}}

\newcommand{\B}{{\mathcal{B}}}
\newcommand{\U}{{\mathcal{U}}}

\newcommand{\T}{{\mathbb{T}}}
\newcommand{\D}{{\mathcal{D}}}

\newcommand{\Ot}{{\mathcal{O}_2}}
\newcommand{\OI}{{\mathcal{O}_{\I}}}

\newcommand{\id}{{\mathrm{id}}}

\newcommand{\diag}{{\mathrm{diag}}}

\newcommand{\Aut}{{\mathrm{Aut}}}

\newcommand{\Ad}{{\mathrm{Ad}}}

\newcommand{\Cu}{{\mathrm{Cu}}}

\newcommand{\dirlim}{\varinjlim}
\newcommand{\ifo}{if and only if }

\newcommand{\ca}{$C^*$-algebra}

\newcommand{\uca}{unital $C^*$-algebra}

\newcommand{\fd}{finite dimensional}

\newcommand{\Rp}{Rokhlin property}

\newcommand{\I}{\infty}

\title[]{Compact group actions with the Rokhlin property}

\date{\today}

\thanks{}
\author[Eusebio Gardella]{Eusebio Gardella}
\address{Eusebio Gardella\\
Westf\"{a}lische Wilhelms-Universit\"{a}t M\"{u}nster, Fachbereich
Mathematik, Einsteinstrasse 62, 48149 M\"{u}nster, Germany}
\email{gardella@uni-muenster.de}
\urladdr{http://wwwmath.uni-muenster.de/u/gardella/}


\begin{document}

\begin{abstract}
We provide a systematic and in-depth study of compact group actions with the Rokhlin property. 
It is show that the Rokhlin property is generic in some cases of interest; the case of totally
disconnected groups being the most interesting one. One of our main results asserts that
the inclusion of the
fixed point algebra induces an order-embedding on $K$-theory, and that it has a splitting
whenever it is restricted to finitely generated subgroups.

We develop new results in the context 
of equivariant semiprojectivity to study actions with the Rokhlin property.  
For example, we characterize when the translation action of a compact group on itself is equivariantly
semiprojective. As an application, it is shown that every Rokhlin action
of a compact Lie group of dimension at most one is a dual action. Similarly, for an action of a compact Lie
group $G$ on $C(X)$, the Rokhlin property is equivalent to freeness together with triviality of the 
principal $G$-bundle $X\to X/G$.
\end{abstract}

\maketitle
\tableofcontents

\renewcommand*{\thetheoremintro}{\Alph{theoremintro}}
\section{Introduction}

The Rokhlin property for discrete group actions on $C^*$-algebras first appeared, at least implicitly, 
in the late 1970's and early 1980's
Fack and Mar\'echal \cite{FacMar_classification_1979}, and Herman and Jones \cite{HerJon_period_1982},
on cyclic group actions on UHF-algebras, as well as in the work of Herman and Ocneanu
\cite{HerOcn_spectral_1986} on integer actions on UHF-algebras.
In \cite{Izu_finiteI_2004}, Izumi provided a formal definition of the Rokhlin property for an action of a finite group.
His classification theorems \cite{Izu_finiteI_2004, Izu_finiteII_2004} are among the major
classification results of finite group actions. These results were extended to the nonunital case
in \cite{GarSan_equivariant_2016}.

Later on, Hirshberg and Winter introduced in
\cite{HirWin_rokhlin_2007} a notion of the Rokhlin property for actions of arbitrary second-countable, compact groups
on \ca s. Crossed products by compact group actions with the Rokhlin property were studied by a number of authors;
see, for example, 
\cite{Izu_finiteI_2004, HirWin_rokhlin_2007, OsaPhi_crossed_2012, Phi_tracial_2011, San_crossed_2015, Gar_crossed_2014}. 

The purpose of this paper is to provide a systematic study of compact group actions with the Rokhlin property, in
a spirit similar to the one in \cite{Gar_rokhlin_2017}, where we explored compact group actions with finite Rokhlin dimension.
We hope that the results contained in this paper will provide the necessary
technical tools to attack problems in which the Rokhlin property of a compact group action can be proved to have a
relevant role. For example, the Rokhlin property for a particular action of $\Z_3$ was used in \cite{PhiVio_simple_2013}
to compute the Elliott invariant and the Cuntz semigroup of a certain simple separable exact \ca\
not anti-isomorphic to itself. More recently, Rokhlin actions of totally disconnected groups were used in
\cite{Gar_automatic_2017} to
construct uncountable many non-cocycle conjugate actions of groups with property (T) on UHF-algebras.

The approach used in this work yields new information even for actions of finite groups.
Indeed, most of our results, particularly those in Sections~3 and~4, had not been noticed
even in the well-studied case of finite groups.

\vspace{0.3cm}

This paper is structured as follows. In Section 2, we recall the definition of the Rokhlin property for a compact
group action, and establish some permanence properties. We also construct examples of compact
group actions with the Rokhlin property on several simple \ca s, and show the following:

\begin{thmintro} (See \autoref{cor:generic}). 
For compact totally disconnected group actions on UHF-absorbing \ca s, and
for compact group actions on $\Ot$-absorbing \ca, the Rokhlin property is generic.
\end{thmintro}

In Section 3, we study the $K$-theory and Cuntz semigroup of fixed point algebras and crossed products, as well
as the equivariant $K$-theory. For example, the augmentation ideal in the group ring annihilates the equivariant
$K$-theory of a compact group action with the Rokhlin property; see \autoref{thm:DiscrKthy}.
We also show the following:

\begin{thmintro} (See \autoref{thm:InjK-Thy}). 
If $\alpha\colon G\to\Aut(A)$ has the Rokhlin property, then 
$K_\ast(A\rtimes_\alpha G)$ can be naturally identified, as a partially ordered group,
with a subgroup of $K_\ast(A)$. The inclusion $K_\ast(A\rtimes_\alpha G)\hookrightarrow K_\ast(A)$
is locally split, and the local splitting is natural with respect to certain maps.
\end{thmintro}

A similar result is obtained for the Cuntz semigroup. These results,
as well as the main technical device used to prove them (\autoref{thm:ApproxHomFixingFPA}), seem not to have been
noticed even in the context of finite group actions (with the exception of part (1) in \autoref{thm:InjK-Thy}, which
was proved in the simple case by Izumi; see Theorem~3.13 in~\cite{Izu_finiteI_2004}).

In Section 4, we explore duality and other applications involving equivariant semiprojectivity. As the main technical
result of independent interest, we characterize those abelian compact groups that are equivariantly semiprojective:

\begin{thmintro} (See \autoref{thm:C(G)eqsj}). Let $G$ be an abelian compact group.
Then $C(G)$ is $G$-equivariantly semiprojective if and only if $G$ is a Lie group and $\dim(G)\leq 1$. \end{thmintro}

As an application, we have:

\begin{corintro} (See \autoref{thm:RokAreDual}).
Let $G$ be a compact abelian Lie group with $\dim(G)\leq 1$. Then every action of $G$ with the Rokhlin property is 
a dual action.
\end{corintro}

This corollary had also not been noticed in the context of finite group actions. For finite abelian group actions on
Kirchberg algebras, this result can be used to give a simpler proof of the classification theorems in \cite{Izu_finiteII_2004}.

As a further application of the techniques in this section, we characterize those topological dynamical systems 
with the Rokhlin property:

\begin{thmintro} (See \autoref{thm:CommSysts}). Let $G$ be an compact Lie group, and let $G\curvearrowright X$
be an action on a compact Hausdorff space $X$. Then the induced action on $C(X)$ has the Rokhlin property
if and only if $G\curvearrowright X$ is free and the principal $G$-bundle $X\to X/G$ is trivial. \end{thmintro}

In this paper, we take $\N=\{1,2,\ldots\}$. For $n\in\N$ with $n\geq 2$, we denote by $\Z_n$ the cyclic group of order
$n$. The circle group will be denoted by $\T$. Lie groups are not be assumed to be connected; in particular, finite
groups are Lie groups. Groups will be assumed to be second-countable, and in particular metrizable.

\vspace{0.3cm}

\indent \textbf{Acknowledgements.} The author is grateful to Chris Phillips for
a number of conversations and feedback, and to Hannes Thiel for many enlightening 
conversations on (equivariant) semiprojectivity. He also thanks Siegfried Echterhoff for calling the author's
attention to the reference \cite{Lan_duality_1979}, and Tron Omland for helpful email correspondence regarding coactions.
Finally, the author thanks the referee for their thorough reading of the manuscript, and for giving very valuable
feedback. 

\vspace{0.3cm}

This paper constitutes essentially Chapter~IV of my PhD dissertation \cite{Gar_thesis_2015}, and a preliminary version has circulated as a
draft for more than two years. Since then, part~(1) of \autoref{thm:InjK-Thy} has also been proved in \cite{BarSza_sequentially_2016}.

\section{Permanence properties, examples and genericity}

In this section, we recall the definition of the Rokhlin property for a compact group action (\autoref{df:Rp}), and present a 
number of permanence properties (\autoref{thm: permanence properties}) that will allow us to construct, in Subsection~2.1, 
interesting examples of Rokhlin actions on simple, nuclear \ca s. Subsection~2.2 contains the main result of this section,
which shows that the Rokhlin property
is in some cases generic among all actions of a compact group $G$ (\autoref{cor:generic}); the case of totally disconnected
groups being the most interesting one.

We begin by introducing some useful notation and terminology.

\begin{df}\label{df:SeqAlgs}
Let $A$ be a \uca. Let $\ell^\I(\N,A)$ denote the set of all bounded sequences in $A$ with the supremum norm
and pointwise operations. Then $\ell^\I(\N,A)$ is a \uca, the unit being the constant sequence 1. Let
$$c_0(\N,A)=\{(a_n)_{n\in\N}\in\ell^\I(\N,A)\colon \lim_{n\to\I}\|a_n\|=0\}.$$
Then $c_0(\N,A)$ is an ideal in $\ell^\I(\N,A)$, and we denote the quotient
$\ell^\I(\N,A)/c_0(\N,A)$
by $A_\I$. We write $\kappa_A\colon \ell^\I(\N,A)\to A_\I$ for the quotient map. We identify $A$ with the unital subalgebra of 
$\ell^\I(\N,A)$ consisting of the constant sequences, and with a unital subalgebra of $A_\I$ by taking its image under $\kappa_A$. 
We write $A_\I\cap A'$ for the relative commutant of $A$ inside of $A_\I$.

\vspace{0.3cm}

If $\alpha\colon G\to\Aut(A)$ is an action of $G$ on $A$, there are actions
of $G$ on $\ell^\I(\N,A)$ and on $A_\I$, which we denote $\alpha^\I$ and $\alpha_\I$, respectively.
Note that
\[(\alpha_\I)_g(A_\I\cap A')\subseteq A_\I\cap A',\]
for all $g\in G$, so that $\alpha_\I$ restricts to an action on $A_\I\cap A'$, also denoted by $\alpha_\I$.

When $G$ is not discrete, these actions are not necessarily continuous:

\begin{eg}
Let $\alpha\colon \T\to \Aut(C(\T))$ be the action induced by left translation. For $n\in\N$,
let $u_n\in C(\T)$ be the unitary given by $u_n(\zeta)=\zeta^n$ for all $\zeta\in\T$. Set
$u=(u_n)_{n\in\N}\in \ell^\I(\N,C(\T))$. It is not difficult to check that the assignments
\[\zeta\mapsto (\alpha^\I)_\zeta(u) \ \mbox{ and } \zeta\mapsto (\alpha_\I)_\zeta(u),\]
are not continuous as a maps $\T\to \ell^\I(\N,C(\T))$ and $\T\to C(\T)_\I=C(\T)_\I\cap C(\T)'$,
respectively. We leave the details to the reader.
\end{eg}

To remedy this issue, we set
$$\ell^\I_\alpha(\N,A)=\{a\in \ell^\I(\N,A)\colon g\mapsto (\alpha_\I)_g(a) \ \mbox{ is continuous}\},$$
and $A_{\I,\alpha}=\kappa_A(\ell^\I_\alpha(\N,A))$. By construction, $A_{\I,\alpha}$ is invariant under $\alpha_\I$,
and the restriction of $\alpha_\I$ to $A_{\I,\alpha}$, which we also denote by $\alpha_\I$, is continuous.\end{df}

If $G$ is a locally compact group, we denote by
$\verb'Lt'\colon G\to\Aut(C_0(G))$ the action induced by left translation of $G$ on itself.
In some situations (particularly in \autoref{thm:CommSysts}),
we make a slight abuse of notation and also denote by $\texttt{Lt}$ the action of $G$ on itself by left
translation.

The following is essentially Definition~3.2 of \cite{HirWin_rokhlin_2007}, except that we do not require the
map $\varphi$ to be injective. However, this condition is automatic: the kernel of $\varphi$ is a
translation invariant ideal in $C(G)$, so it must be either $\{0\}$ or all of $C(G)$.

\begin{df}\label{df:Rp}
Let $A$ be a \uca, let $G$ be a second-countable compact group, and let $\alpha \colon G \to \Aut(A)$ be a continuous action. 
We say that $\alpha$ has the \emph{Rokhlin property} if there is an equivariant unital
homomorphism
\[\varphi\colon (C(G),\texttt{Lt})\to (A_{\I,\alpha}\cap A',\alpha_\I).\]\end{df}

Circle actions with the Rokhlin property have been studied in \cite{Gar_classificationI_2014}, \cite{Gar_classificationII_2014}, and \cite{Gar_circle_2014}.
There, the definition given reads as follows: An action $\alpha\colon \T\to\Aut(A)$ of the circle $\T$ on a unital
\ca\ $A$ is said to have the Rokhlin property if for every finite subset $F\subseteq A$ and for every $\ep>0$,
there exists a unitary $u$ in $A$ such that
\be\item $\|\alpha_\zeta(u)-\zeta u\|<\ep$ for all $\zeta\in\T$, and
\item $\|au-ua\|<\ep$ for all $a\in F$.\ee

For the sake of consistency, and since we will use results from \cite{Gar_classificationI_2014}, we check that both definitions agree.
The main point is to show that the estimate in (1) above is uniform on $\zeta\in\T$.

\begin{prop}
Let $A$ be a separable, unital \ca, and let $\alpha\colon \T\to\Aut(A)$ be a continuous action. Then $\alpha$
has the Rokhlin property in the sense of \autoref{df:Rp} if and only if it has the Rokhlin property in the sense
of \cite{Gar_classificationI_2014}.
\end{prop}
\begin{proof}
Assume that an action $\alpha\colon\T\to\Aut(A)$ has the Rokhlin property in the sense of \cite{Gar_classificationI_2014}.
Since unital homomorphisms from $C(\T)$ into a \uca\ are in one-to-one correspondence with unitaries in the \ca,
it is clear that there is a unital homomorphism $\varphi\colon C(\T)\to A_\I$. By condition (2), the image of this homomorphism is
contained in the commutant of $A$. Moreover, the fact that $\|\alpha_\zeta(u)-\zeta u\|<\ep$ is arbitrarily small for
every fixed $\zeta\in\T$ implies that $\varphi$ is equivariant. Finally, the uniformity condition implies that the image
of $\varphi$ is contained in $A_{\alpha,\I}$. This verifies \autoref{df:Rp}.

To prove the converse implication, let $F\subseteq A$ be
a finite subset and let $\ep>0$. Let
$u\in A_{\I,\alpha}\cap A'$ be a unitary inducing a unital homomorphism as in \autoref{df:Rp} for $G=\T$.
Choose a sequence $(u_n)_{n\in\N}$ of unitaries in $A$ such that
\[\kappa_A((u_n)_{n\in\N})=u.\]

Since $u$ belongs to the commutant of $A$, we have
\beq\label{eq:1} \lim_{n\to\I}\|au_n-u_na\|=0\eeq
for all $a\in A$.

On the other hand, the fact that $(\alpha_\I)_\zeta(u)=\zeta u$ for all $\zeta\in \T$, shows that
\beq \label{eq:2} \lim_{n\to\I}\|\alpha_\zeta(u_n)-\zeta u_n\|=0\eeq
for all $\zeta\in\T$. This by itself does not imply that we can choose $n$ large enough so that
$\|\alpha_\zeta(u_n)-\zeta u_n\|<\ep$ holds for \emph{all} $\zeta\in \T$. Put in a different way,
one needs to show that the sequence of functions $f_n\colon \T\to \R$ given by
\[f_n(\zeta)=\|\alpha_\zeta(u_n)-\zeta u_n\|\]
for $\zeta\in\T$, converges \emph{uniformly} to zero. Equation (\ref{eq:2}) implies that $(f_n)_{n\in\N}$
converges pointwise to zero. Without loss of generality, we may assume that $f_n(\zeta)\geq f_{n+1}(\zeta)$
for all $n\in\N$ and all $\zeta\in\T$.

Since $u$ belongs to $A_{\alpha,\I}$, it follows that $f_n$ is continuous for
all $n\in\N$. Now, by Dini's theorem,
a decreasing sequence of continuous functions converges uniformly
if it converges pointwise to a continuous function.
Thus, for the finite set $F\subseteq A$ and the tolerance $\ep>0$ given,
we use this fact together with Equation (\ref{eq:1}) to find $n_0\in\N$ such that
$\|\alpha_\zeta(u_{n_0})-\zeta u_{n_0}\|<\ep$ for all $\zeta\in\T$, and
$\|au_{n_0}-u_{n_0}a\|<\ep$ for all $a\in F$.
\end{proof}

Since unital completely positive maps of order zero are necessarily homomorphisms, it is easy to see
that the Rokhlin property for a compact group action agrees with Rokhlin dimension zero in the sense of
Definition~3.2 in \cite{Gar_rokhlin_2017}. In particular, the following is a consequence of Theorem~3.8 in
\cite{Gar_rokhlin_2017} and Proposition~5.1 in~\cite{Gar_regularity_2014}.

\begin{thm}\label{thm: permanence properties}
Let $A$ be a unital \ca\, let $G$ be a second-countable compact group, and let $\alpha\colon G\to\Aut(A)$
be a continuous action of $G$ on $A$.
\be
\item Let $B$ be a unital \ca, and let $\beta\colon G\to\Aut(B)$ be a continuous action of $G$ on $B$.
Let $A\otimes B$ be any \ca\ completion of the algebraic tensor product of $A$ and $B$ for which the tensor
product action $g\mapsto (\alpha\otimes\beta)_g=\alpha_g\otimes\beta_g$ is defined. If $\alpha$ has the
Rokhlin property, then so does $\alpha\otimes \beta$.
\item Let $I$ be an $\alpha$-invariant ideal in $A$, and denote by $\overline{\alpha}
\colon G\to \Aut(A/I)$ the induced action on $A/I$. If $\alpha$ has the Rokhlin property, then so
does $\overline{\alpha}$.
\item Suppose that $\alpha$ has the Rokhlin property and let $p$ be an $\alpha$-invariant projection in
$A$. Set $B=pAp$ and denote by $\beta\colon G\to\Aut(B)$ the compressed action of $G$. Then $\beta$ has
the Rokhlin property.
\ee
Furthermore,
\be
\setcounter{enumi}{3}
\item Let $(A_n,\iota_n)_{n\in\N}$ be a direct system of unital \ca s with unital
connecting maps, and for each $n\in\N$, let $\alpha^{(n)}\colon G\to\Aut(A_n)$ be a
continuous action such that $\iota_n\circ\alpha^{(n)}_g=\alpha^{(n+1)}_g\circ\iota_n$ for all $n\in\N$ and all
$g\in G$. Suppose that $A=\varinjlim A_n$ and that $\alpha=\varinjlim \alpha^{(n)}$. If $\alpha^{(n)}$ has
the Rokhlin property for infinitely many values of $n$, then $\alpha$ has the Rokhlin property as well.\ee \end{thm}

It is not in general the case that the Rokhlin property for compact group actions is preserved by restricting
to a closed subgroup. The reader is referred to Section~3.1 in \cite{Gar_rokhlin_2017} for a
discussion about the interaction between Rokhlin dimension and restriction to
closed subgroups.

\subsection{Examples and non-existence}

Compact group actions with the Rokhlin property are rare (and they seem to be even less
common if the group is connected). In a forthcoming paper (\cite{Gar_automatic_2017}), we will show that there are many
\ca s of interest that do not admit any non-trivial compact group action with the Rokhlin
property (such as the Cuntz algebra $\mathcal{O}_\I$ and the Jiang-Su algebra $\mathcal{Z}$;
see \cite{HirPhi_rokhlin_2015} for a stronger statement valid for compact Lie groups), while there are many \ca s
that only admit actions with the Rokhlin property of \emph{totally disconnected} compact
groups, such as
the Cuntz algebras $\mathcal{O}_n$ for $n\geq 3$, UHF-algebras, or even AI-algebras.
See \cite{Gar_circle_2014} and \cite{Gar_classificationI_2014} for some non-existence results of circle actions
with the \Rp.

\vspace{0.3cm}

In this section, we shall construct a family of examples of compact group actions with the Rokhlin property
on certain simple AH-algebras of no dimension growth, and on certain Kirchberg algebras,
including $\mathcal{O}_2$. We also show
that compact group actions on $\Ot$-absorbing \ca s are generic, in a suitable sense; see
\autoref{Rokhlin are generic on D-absorbing algs}. The next one is the canonical example of 
a Rokhlin action.

\begin{eg}\label{eg:Lt}
Given a second-countable compact group $G$, the action
$\texttt{Lt}\colon G\to \Aut(C(G))$ has the \Rp, essentially by definition.\end{eg}

For the theory to be applicable in cases of interest, we must exhibit Rokhlin actions on \ca s that fit within
the classification program of Elliott. Our next three examples do this. 

\begin{eg}\label{eg:RpAH}
Let $G$ be a second-countable compact group. For $n\in\N$, set $A_n=C(G)\otimes M_{2^n}$. Set
$\alpha^{(n)}=\texttt{Lt}\otimes\id_{M_{2^n}}\colon G\to\Aut(A_n)$. Then $\alpha^{(n)}$ has the
Rokhlin property by part (1) of \autoref{thm: permanence properties} and \autoref{eg:Lt}.
Fix a countable subset $X=\{x_1,x_2,x_3,\ldots\}$ of
$G$ such that $\{x_m,x_{m+1},\ldots\}$ is dense in $G$ for all $m\in\N$.
Given $n\in\N$, define a map $\iota_n\colon A_n\to
A_{n+1}$ by
$$\iota_n(f)=\left(
               \begin{array}{cc}
                 f & 0 \\
                 0 & \texttt{Lt}_{x_n}(f) \\
               \end{array}
             \right)$$
for every $f$ in $A_n$. Then $\iota_n$ is unital and injective. The direct limit
$A=\dirlim(A_n,\iota_n)$ is clearly a unital AH-algebra of no dimension growth, and it is simple by
Proposition~2.1 in \cite{DadNagNemPas_reduction_1992}.
It is easy to check that
$$\iota_n\circ\alpha^{(n)}_g=\alpha^{(n+1)}_g\circ\iota_n$$
for all $n\in\N$ and all $g\in G$, and hence
$\left(\alpha^{(n)}\right)_{n\in\N}$ induces a direct limit action
$\alpha=\varinjlim \alpha^{(n)}$ of $G$ on $A$. Then $\alpha$ has the
Rokhlin property by part (3) of \autoref{thm: permanence properties}.

The (graded) $K$-theory of $A$ is easily seen to be
$K_\ast(A)\cong K_\ast(C(G))\otimes\Z\left[\frac{1}{2}\right]$.
Additionally, $A^\alpha$ is isomorphic to the CAR algebra, and the inclusion
$A^\alpha\to A$, at the level of $K_0$, induces the canonical embedding
$\Z\left[\frac{1}{2}\right]\to K_0(C(G))\otimes\Z\left[\frac{1}{2}\right]$ as the
second tensor factor.  \end{eg}

In \autoref{eg:RpAH}, the $2^\I$ UHF pattern can be replaced by any other UHF or (simple)
AF pattern, and the resulting \ca\ is also a (simple) AH-algebra with no dimension growth.
If the group is totally disconnected, the direct limit algebra will be an AF-algebra.
For non-trivial groups, these AF-algebras will nevertheless not be UHF-algebras, even if a
UHF pattern is followed.

\begin{eg}\label{eg:onPI}
Given a second-countable compact group $G$, let $A$ and $\alpha$ be as in \autoref{eg:RpAH}.
Then
\[\alpha\otimes\id_{\OI}\colon G\to \Aut(A\otimes\OI)\]
has the \Rp\ by part (1) of \autoref{thm: permanence properties}, and $A\otimes\OI$ is a Kirchberg
algebra. One can obtain actions of $G$ on other Kirchberg algebras by following a different
UHF or AF pattern in \autoref{eg:RpAH}.\end{eg}

\begin{eg}\label{eg:RpOt}
Let $G$ be a second-countable compact group, and let $A$ and $\alpha$ be as in \autoref{eg:RpAH}.
Use Theorem~3.8 in \cite{KirPhi_embedding_2000} to choose an isomorphism $\varphi\colon A\otimes \Ot\to \Ot$, 
and define an action $\beta\colon G\to\Aut(\Ot)$ by $\beta_g=\varphi\circ (\alpha_g\otimes \id_{\Ot})\circ \varphi^{-1}$ for $g\in G$.
Then $\beta$ has the \Rp\ by part (1) of \autoref{thm: permanence properties}.

More generally, the action constructed in \autoref{eg:RpAH} can be used to construct an action of $G$ on
any $\mathcal{O}_2$-absorbing \ca. \end{eg}

In contrast, it follows from the following proposition that only finite groups act with the Rokhlin
property on finite dimensional \ca s (and, in this case, the action must be a permutation of the simple
summands). The result is not surprising, but our proof allows us to
show that the result is true even under the much weaker assumption that the action be pointwise
outer. In particular, by Theorem~4.14 in \cite{Gar_rokhlin_2017}, this applies to compact group actions with
finite Rokhlin dimension.

\begin{prop}
Let $G$ be a compact group, let $A$ be a finite dimensional \ca, and let $\alpha\colon G\to\Aut(A)$
be a continuous action. If $\alpha_g$ is outer for all $g\in G\setminus\{1\}$, then $G$ must be finite.
\end{prop}
\begin{proof}
Choose positive integers $m, n_1,\ldots,n_m$ such that $A\cong\bigoplus\limits_{j=1}^m M_{n_j}$.
Denote by $G_0$ the connected component of the identity of $G$. By Proposition~3.9 in \cite{Gar_circle_2014}, the
restriction $\alpha|_{G_0}$ acts trivially on $K$-theory. This is easily seen to be equivalent to
$\alpha_g(M_{n_j})=M_{n_j}$ for all $g\in G_0$ and all $j=1,\ldots,m$.
Given $g\in G_0$, and since every automorphism of a matrix algebra is inner,
it follows that $\alpha_g|_{M_{n_j}}$ is inner
for all $j=1,\ldots,m$. Hence $\alpha_g$ is inner, and we conclude that $G_0=\{1\}$. In other
words, $G$ is totally disconnected.

If $G$ is infinite, then there is a strictly decreasing sequence
\[G\supseteq H_1\subseteq H_2\supseteq\cdots \supseteq \{1\}\]
of infinite closed subgroups of $G$ with finite index. Therefore there is an increasing sequence of inclusions
\[A^G\subseteq A^{H_1}\subseteq \cdots \subseteq A^{\{1\}}=A.\]
Since $A$ is finite dimensional, this sequence must stabilize, and hence there exists an infinite subgroup $H$
of $G$ such that $A^H=A$. We conclude that $\alpha_g=\id_A$ for all $g\in H$. This is a contradiction, and
hence $G$ is finite.
\end{proof}

The following technical definition will be needed in the next subsection.

\begin{df} \label{df:StrongRp}
Let $A$ be a \uca, let $G$ be a second-countable compact group, and let $\alpha\colon G\to\Aut(A)$ be a continuous
action. We say that $\alpha$ has the \emph{strong Rokhlin property} if there exists a unital equivariant homomorphism
\[\varphi\colon (C(G),\texttt{Lt})\to (A_{\I,\alpha}\cap A',\alpha_\I)\]
which can be lifted to an equivariant homomorphism $\varphi\colon (C(G),\texttt{Lt})\to (\ell^\I_\alpha(A),\alpha^{\I})$.
\end{df}

The difference between the Rokhlin property (\autoref{df:Rp}) and the strong Rokhlin property, is that in
\autoref{df:Rp}, the homomorphism $\varphi$ is not assumed to be liftable to a homomorphism into $\ell^\I_\alpha(A)$. 
The following remark highlights this difference in terms of elements in the coefficient algebra $A$:

\begin{rem} Let $\alpha\colon G\to\Aut(A)$ be an action of a second-countable compact group on a unital
\ca\ $A$. The following equivalences
are easy to check: 
\be
\item[(a)] $\alpha$ has the Rokhlin property if and only if for every finite subset $F\subseteq A$,
for every finite subset $S\subseteq C(G)$, and for every $\ep>0$, there exists a unital equivariant
\emph{completely positive contractive} map $\varphi\colon C(G)\to A$ such that
$\|\psi(f)a-a\psi(f)\|<\ep$ and $\|\psi(fh)-\psi(f)\psi(h)\|<\ep$
for all $f,h\in S$ and for all $a\in F$.
\item[(b)] $\alpha$
has the strong Rokhlin property
if and only if for every finite subset $F\subseteq A$,
for every finite subset $S\subseteq C(G)$, and for every $\ep>0$, there exists a unital equivariant
\emph{homomorphism} $\psi\colon C(G)\to A$ such that
$\|\psi(f)a-a\psi(f)\|<\ep$
for all $f\in S$ and for all $a\in F$.
\ee
(The difference between both conditions is that, while $\psi$ is only assumed to be \emph{almost} multiplicative
in part (a), it is a homomorphism in part (b).)
\end{rem}

By Proposition~3.3 in \cite{Gar_classificationI_2014}, any action of the circle with the Rokhlin
property has the strong
Rokhlin property. More generally, it will
follow from
\autoref{thm:C(G)eqsj} that the Rokhlin property is equivalent to the strong Rokhlin property for (abelian) compact
Lie groups with $\dim(G)\leq 1$. We do not have an example of an action with the Rokhlin property that does not
have the strong Rokhlin property, although we suspect it exists. The following will be needed later.

\begin{lma}\label{lma:OtStrRp}
Let $G$ be a second-countable compact group.
Then there exists a continuous
action $\alpha\colon G\to\Aut(\Ot)$ with the strong Rokhlin property.
\end{lma}
\begin{proof}
By Theorem~3.8 in \cite{KirPhi_embedding_2000}, it is enough to construct an action of $G$ on a simple,
unital, separable, nuclear \ca, with the strong Rokhlin property. It is straightforward
to check that the action constructed in \autoref{eg:RpAH} satisfies the desired condition. We omit
the details.
\end{proof}

\subsection{The Rokhlin property is generic.}
In this subsection, we study genericity of Rokhlin actions.
We show that if $A$ is a separable, unital, UHF-absorbing \ca,
and if $G$ is a totally disconnected compact group, then
then $G$-actions on $A$ with the Rokhlin property are generic.
For an arbitrary compact group, a similar result is obtained for $\Ot$-absorbing
\ca s; see \autoref{cor:generic}.

Throughout, $A$ will be a separable, unital \ca, and $G$ will be a second-countable compact group.

\begin{df} Given an enumeration $X=\{a_1,a_2,\ldots\}$ of a countable dense subset of the unit ball of $A$, define
\[\rho_X^{(0)}(\alpha,\beta)=\sum\limits_{k=1}^\I \frac{\|\alpha(a_k)-\beta(a_k)\|}{2^k}\]
and
\[\rho_X(\alpha,\beta)=\rho_X^{(0)}(\alpha,\beta)+\rho_X^{(0)}(\alpha^{-1},\beta^{-1})\]
for $\alpha,\beta\in \Aut(A)$.\end{df}

Denote by $\mbox{Act}_G(A)$ the set of all continuous actions of $G$ on $A$, and define
\[\rho_{G,S}(\alpha,\beta)=\max_{g\in G} \rho_X(\alpha_g,\beta_g),\]
for $\alpha,\beta\in \mbox{Act}_G(A)$. We record the following standard fact.

\begin{lma}\label{space of T-actions is complete}
For any enumeration $X$ as above,
the function $\rho_{G,X}$ is a complete metric on $\mbox{Act}_G(A)$.\end{lma}
%\begin{proof} Let $\left(\alpha^{(n)}\right)_{n\in \N}$ be a Cauchy sequence in $\mbox{Act}_G(A)$, so that for every $\varepsilon>0$
%there is $n_0 \in \N$ such that, for every $n,m\geq n_0$, we have $\rho_{G,X}\left(\alpha^{(n)},\alpha^{(m)}\right)<\varepsilon$. We
%want to show that there is $\alpha\in\mbox{Act}_G(A)$ such that $\lim\limits_{n\to\I}\rho_{G,X}\left(\alpha,\alpha^{(n)}\right)= 0$.

%Given $g\in G$, we have $\rho_X\left(\alpha_g^{(n)},\alpha_g^{(m)}\right)\leq \rho_{G,X}\left(\alpha^{(n)},\alpha^{(m)}\right)$,
%and hence $\left(\alpha_g^{(n)}\right)_{n\in \N}$ is Cauchy in $\Aut(A)$. By Lemma 3.2 in \cite{Phi_generic},
%the pointwise norm limit of the sequence $\left(\alpha^{(n)}_g\right)_{n\in \N}$ exists, and we denote
%it by $\alpha_g$. It
%also follows from Lemma 3.2 in \cite{Phi_generic} that $\alpha_g$ is an automorphism of $A$, with inverse
%$\alpha_{g^{-1}}$. Moreover, the map $\alpha\colon G\to\Aut(A)$ given by $g\mapsto \alpha_g$, is a group homomorphism, since
%it is the pointwise norm limit of group homomorphisms. It remains to check that it is continuous, and this follows using an $\frac{\ep}{3}$
%argument from the equation $\lim\limits_{n\to\I}\left\|\alpha_g^{(n)}(a_k)-\alpha_g(a_k)\right\|= 0$ for all $k$ in $\N$, and the fact that
%$\alpha^{(n)}\colon G\to\Aut(A)$ is continuous for all $n\in\N$. We omit the details.\end{proof}

\begin{nota} Given a finite subset $F\subseteq A$, given a finite subset $S\subseteq C(G)$, and given $\varepsilon>0$,
we let $W_G(F,S,\varepsilon)$ denote the set of all actions $\alpha\in\mbox{Act}_G(A)$ such that there exists
a unital completely positive linear map $\varphi\colon C(G)\to A$ satisfying
\be\item $\|\varphi(f)a-a\varphi(f)\|<\varepsilon$ for all $a\in F$ and for all $f\in S$;
\item $\|\varphi(f_1f_2)-\varphi(f_1)(\varphi(f_2)\|<\ep$ for all $f_1,f_2\in S$; and
\item $\|\alpha_g(\varphi(f))-\varphi(\texttt{Lt}_g(f))\|<\varepsilon$ for all $g\in G$ and for all $f\in S$.\ee\end{nota}

It is easy to check that an action $\alpha\in \mbox{Act}_G(A)$ has the \Rp\ \ifo it belongs to $W_G(F,S,\varepsilon)$ for all finite subsets
$F\subseteq A$ and $S\subseteq C(G)$, and for all positive numbers $\varepsilon>0$.

If $Z$ is a set, we denote by $\mathcal{P}_f(Z)$ the set of all finite subsets of $Z$. Note that $|\mathcal{P}_f(Z)|=|Z|$ if $Z$ is
infinite.

\begin{lma}\label{lma} Let $X$ be a countable dense subset of the unit ball of $A$, and let $Y$ be a countable
dense subset of the unit ball of $C(G)$.
Then $\alpha\in\mbox{Act}_G(A)$ has the \Rp\ \ifo it belongs to the countable intersection
$$\bigcap_{F\in \mathcal{P}_f(X)}\bigcap_{S\in \mathcal{P}_f(Y)}\bigcap_{n=1}^\I W_G\left(F,S,\frac{1}{n}\right).$$\end{lma}
\begin{proof} One just needs to approximate any finite subset of $A$ by scalar multiples of elements in a finite subset of $X$,
and likewise for finite subsets of $C(G)$. (We are implicitly using that both $A$ and $C(G)$ are separable.)
We omit the details.\end{proof}

Since the following proposition is a bit technical, we give a sketch of its proof in a simplified setting. Adopt the notation and
assumptions of the proposition, and for the sake of the
argument, suppose that there exists an isomorphism $\theta\colon A\otimes \D \to A$ such that $a\mapsto (\theta(a\otimes 1_{\D})$ 
is unitarily equivalent to $\id_A$, and let $w\in \U(A)$ be an implementing unitary. Set $\rho=\Ad(w)\circ\theta$. One can check that
if $\alpha\in \mathrm{Act}_G(A)$, then 
\[\alpha_g=\rho\circ(\alpha_g\circ\gamma_g)\circ \rho^{-1}\]
for all $g\in G$. Since $\gamma$ has the Rokhlin property, so does $\alpha\otimes\gamma$, and by the above also $\alpha$. We deduce
that $\alpha$ belongs to $W_G(F,S,\varepsilon)$ for all $F,S,\varepsilon$. In general, when $a\mapsto\theta(a\otimes 1_{\D})$ is
\emph{approximately} unitarily equivalent to $\id_A$, one can only approximate $\alpha$ by elements in $W_G(F,S,\varepsilon)$,
showing that this set is dense.

For use in the following proposition, we denote by $d\colon G\times G\to\R$ a (left) invariant metric on $G$.

\begin{prop}\label{W(F,epsilon) is dense} Let $A$ and $\D$ be unital, separable \ca s, such that there is an action
$\gamma\colon G\to\Aut(\D)$ with the strong \Rp. Suppose that there exists an isomorphism $\theta\colon A\otimes \D\to A$
such that $a\mapsto \theta(a\otimes 1_\D)$ is approximately unitarily equivalent to $\id_A$. Then for every finite subset
$F\subseteq A$ and every $\varepsilon>0$, the set $W_G(F,S,\varepsilon)$ is open and dense in $\mbox{Act}_G(A)$.\end{prop}
\begin{proof} We first check that $W_G(F,S,\varepsilon)$ is open.
Choose an enumeration $X=\{a_1,a_2,\ldots\}$ of a countable
dense subset of the unit ball of $A$, and an enumeration $Y=\{f_1,f_2,\ldots\}$
of a dense subset of the unit ball of $C(G)$.
Let $\alpha$ in $W_G(F,S,\varepsilon)$, and choose a unital completely positive linear map $\varphi\colon C(G)\to A$
satisfying
\be\item $\|\varphi(f)a-a\varphi(f)\|<\varepsilon$ for all $a\in F$ and for all $f\in S$;
\item $\|\varphi(f h)-\varphi(f)(\varphi(h)\|<\ep$ for all $f,h\in S$; and
\item $\|\alpha_g(\varphi(f))-\varphi(\texttt{Lt}_g(f))\|<\varepsilon$ for all $g\in G$ and for all $f\in S$.\ee

Set
\[\varepsilon_0=\max_{g\in G}\max_{f\in S}\left\|\alpha_g(\varphi(f))-\varphi(\texttt{Lt}_g(f))\right\|,\]
so that $\varepsilon_1=\varepsilon-\varepsilon_0$ is positive.
For $f\in S$, let $k_f\in \N$ satisfy $\|a_{k_f}-\varphi(f)\|<\frac{\varepsilon_1}{3}$.
Set $M=\max\limits_{f\in S} k_f$.

We claim that if
$\alpha'\in \mbox{Act}_G(A)$ satisfies $\rho_{G,X}(\alpha',\alpha)<\frac{\varepsilon_1}{2^{M}3}$, then $\alpha'$
belongs to $W_G(F,S,\varepsilon)$.

Let $\alpha'\in \mbox{Act}_G(A)$ satisfy $\rho_{G,X}(\alpha',\alpha)<\frac{\varepsilon_1}{2^{M}3}$, and
let $f\in S$. Then
\begin{align*}\|\alpha'_g(\varphi(f))-\varphi(\texttt{Lt}_g(f))\|&\leq
\|\alpha'_g(\varphi(f))-\alpha_g(\varphi(f))\|+\|\alpha_g(\varphi(f))-\varphi(\texttt{Lt}_g(f))\|\\
&\leq \frac{2\ep_1}{3} +\|\alpha_g'(a_{k_f})-\alpha_g(a_{k_f})\|+ \varepsilon_0\\
&\leq \frac{2\ep_1}{3} + 2^M\rho_{G,X}(\alpha,\alpha')+\varepsilon_0\\
&=\varepsilon_1+\varepsilon_0=\varepsilon.\end{align*}
The claim is proved. It follows that $W_G(F,S,\varepsilon)$ is open.

We will now show that $W_G(F,S,\varepsilon)$ is dense in $\mbox{Act}_G(A)$.
Let $\alpha$ be an arbitrary action in $\mbox{Act}_G(A)$,
let $E\subseteq A$ be a finite subset, and let $\delta>0$.
We want to find $\beta\in\mbox{Act}_G(A)$ such that $\beta$ belongs to $W_G(F,S,\varepsilon)$
and $\|\alpha_g(a)-\beta_g(a)\|<\delta$ for all $g\in G$ and all $a\in E$.

Fix $0<\delta'<\min\{\delta,\varepsilon\}$. Since $G$ is compact and second-countable, it
admits a left-invariant metric, which we denote by $d\colon G\times G\to \R$.
Since $\alpha$ is continuous, there is $\delta_0>0$ such that
whenever $g,g'\in  G$ satisfy $d(g,g')<\delta_0$, then
\[\|\alpha_{g}(a)-\alpha_{g'}(a)\|<\frac{\dt'}{4}\]
for all $a\in E$.
Choose $m\in\N$ and $g_1,\ldots,g_m\in G$, such that for every $g\in G$, there is $j\in\N$, with $1\leq j\leq m$,
satisfying
$d(g,g_j)<\delta_0$. Choose $w\in\U(A)$ with
\[\|w\theta(1_A\otimes a)w^*-a\|<\frac{\dt'}{2}\]
for all
$a\in E\cup \bigcup\limits_{j=1}^m\alpha_{g_j}(E)$. Set $\rho=\Ad(w)\circ\theta$ and define an action
$\beta\in\mbox{Act}_G(A)$ by
\[\beta_g=\rho\circ(\gamma_g\otimes\alpha_g)\circ\rho^{-1}\]
for $g\in G$.

We claim that $\beta$ belongs to $W_G(F,S,\varepsilon)$. Choose $r\in\N$,
$d_1,\ldots,d_r\in \D$, and $x_1,\ldots,x_r\in A$, such that
$w'=\sum\limits_{\ell=1}^rx_\ell\otimes d_\ell$ satisfies $\|w-w'\|<\frac{\dt}{3}$.
Use the strong Rokhlin property of $\gamma$ to find a unital equivariant homomorphism
$\varphi\colon C(G)\to \D$ such that
\[\|\varphi(f)d_\ell-d_\ell \varphi(f)\|<\frac{\ep}{4}\]
for all $f\in S$ and
for all $\ell=1,\ldots,r$. Then
\[\|(1_A\otimes \varphi(f))w'-w'(1_A\otimes \varphi(f))\|<\frac{\dt}{3}\]
for all $f\in S$, and hence $\|(1_A\otimes \varphi(f))w-w(1_A\otimes \varphi(f))\|<\delta$.
Define a unital homomorphism $\psi\colon C(G)\to A$ by
\[\psi(f)=\theta(1_A\otimes \varphi(f))\]
for $f\in C(G)$. Given $g\in G$ and $f\in S$, we have
\begin{align*}&\|\beta_g(\psi(f))- \psi(\texttt{Lt}_g(f))\| \\
&=\left\|w\theta\left((\alpha_g\otimes\gamma_g)(\theta^{-1}(w^*\theta(1_A\otimes \varphi(f))w))\right)w^*-
\theta(1_A\otimes \varphi(\texttt{Lt}_g(f)))\right\|\\
& \leq \left\|w\theta\left((\alpha_g\otimes\gamma_g)(\theta^{-1}(w^*\theta(1_A\otimes \varphi(f))w))\right)w^*-w\theta\left((\alpha_g\otimes\gamma_g)(1_A\otimes \varphi(f))\right)w^*\right\|\\
& \ \ \ +\left\|w\theta\left((\alpha_g\otimes\gamma_g)(1_A\otimes \varphi(f))\right)w^*-
\theta(1_A\otimes \varphi(\texttt{Lt}_g(f)))\right\|\\
&<\frac{\dt'}{2} + \left\|w\theta\left(g 1_A\otimes \varphi(f)\right)w^*- g\theta(1_A\otimes \varphi(f))\right\|\\
&<\frac{\dt'}{2}+\frac{\dt'}{2}=\delta'<\varepsilon,\end{align*}
 and thus $\|\beta_g(\psi(f))- \psi(\texttt{Lt}_g(f))\|<\varepsilon$ for all $g\in G$
and for all $f\in S$. On the other hand, given $a\in F$ and $f\in S$, we use the identity
\[(a\otimes 1_\D)(1_A\otimes \varphi(f))=(1_A\otimes \varphi(f))(a\otimes 1_\D)\]
at the third step, to obtain
\begin{align*}\|\psi(f)a-a\psi(f)\|& = \|\theta(1_A\otimes \varphi(f))a-a\theta(1_A\otimes \varphi(f))\| \\
&\leq \|\theta(1_A\otimes \varphi(f))a-\theta(1_A\otimes \varphi(f))w\theta(a\otimes 1_\D)w^*\|\\
& \ \ \ +\|\theta(1_A\otimes \varphi(f))w\theta(a\otimes 1_\D)w^*-w\theta(a\otimes 1_\D)w^*\theta(1_A\otimes \varphi(f))\|\\
& \ \ \ +\|w\theta(a\otimes 1_\D)w^*\theta(1_A\otimes \varphi(f))-a\theta(1_A\otimes \varphi(f))\|\\
&<\frac{\dt'}{2}+0+\frac{\dt'}{2}=\delta'<\varepsilon.\end{align*}
This proves the claim.\\
\indent It remains to prove that $\|\beta_g(a)-\alpha_g(a)\|<\delta$ for all $a$ in $E$ and all $ g$ in $ G$.
For fixed $g\in G$, choose $j\in \{1,\ldots,m\}$ such that $d(g,g_j)<\delta_0$. Then, for $a\in E$, we have
\begin{align*}\|\beta_g(a)-\alpha_g(a)\|&=\|w\theta\left((\alpha_g\otimes\gamma_g)(\theta^{-1}(w^*aw))\right)-\alpha_g(a)\| \\
&\leq \|w\theta\left((\alpha_g\otimes\gamma_g)(\theta^{-1}(w^*aw))\right)-w\theta\left((\alpha_g\otimes\gamma_g)(a\otimes 1_\D)\right)\|\\
& \ \ \ + \|w\theta\left((\alpha_g\otimes\gamma_g)(a\otimes 1_\D)\right)-\alpha_g(a)\|\\
&< \frac{\dt'}{2}+\|w\theta(\alpha_g(a)\otimes 1_\D)w^*-\alpha_g(a)\|\\
&\leq \frac{\dt'}{2}+ \|w\theta(\alpha_g(a)\otimes 1_\D)w^*-w\theta(\alpha_{g_j}(a)\otimes 1_\D)w^*\|\\
&\ \ \ +\|w\theta(\alpha_{g_j}(a)\otimes 1_\D)w^*-\alpha_{g_j}(a)\|+\|\alpha_{g_j}(a)-\alpha_{ g}(a)\|\\
&<\frac{\dt'}{2}+\frac{\dt'}{4}+\frac{\dt'}{4}=\delta'<\delta.\end{align*}
This finishes the proof.\end{proof}

\begin{thm}\label{Rokhlin are generic on D-absorbing algs}
Let $A$ and $\D$ be unital, separable \ca s, such that there is an action
$\gamma\colon G\to\Aut(\D)$ with the strong \Rp.
Suppose that there exists an isomorphism $\theta\colon A\otimes \D\to A$ such that
$a\mapsto \theta(a\otimes 1_\D)$ is approximately unitarily equivalent to $\id_A$.
Then the set of actions of $G$ on $A$ with the \Rp\ is a dense
$G_\delta$-set in $\mbox{Act}_G(A)$. \end{thm}
\begin{proof} By \autoref{lma}, the set of all $G$-actions on $A$ that have the \Rp\ is precisely the countable intersection
$$\bigcap_{F\in \mathcal{P}_f(X)}\bigcap_{S\in \mathcal{P}_f(Y)}\bigcap_{n\in\N} W_G\left(F,S,\frac{1}{n}\right).$$
By \autoref{W(F,epsilon) is dense}, each $W_G\left(F,S,\frac{1}{n}\right)$ is open and dense in $\mbox{Act}_G(A)$, which is a complete metric space
by \autoref{space of T-actions is complete}. The result then follows from the Baire Category Theorem.\end{proof}

We make some comments on the verification of the assumptions in the theorem above.
Recall that a unital, separable \ca\ $\D$ is said to be \emph{strongly self-absorbing},
if it is infinite-dimensional and the map $\D\to \D\otimes \D$ given by $d\mapsto d\otimes 1$,
is approximately unitarily equivalent to an isomorphism. Whenever $\D$ is strongly self-absorbing,
and $A$ is a separable, \uca\ satisfying $A\otimes\D\cong A$, then there exists an isomorphism
$\theta\colon A\otimes \D\to A$ satisfying the assumptions of \autoref{Rokhlin are generic on D-absorbing algs};
see, for example, Theorem~7.2.2 in \cite{RorSto_classification_2002}.
On the other hand, finding compact group actions on strongly self-absorbing \ca s with the (strong) Rokhlin
property is not always easy. For example, neither the Jiang-Su algebra $\mathcal{Z}$ nor the Cuntz algebra
$\mathcal{O}_\I$ admit any non-trivial compact group action with the Rokhlin property; see \cite{Gar_automatic_2017}
(the case of Lie groups was proved in \cite{HirPhi_rokhlin_2015}).
At the opposite end, we already saw in \autoref{lma:OtStrRp} that \emph{every} compact group admits an action on $\mathcal{O}_2$
with the strong Rokhlin property.

Finding examples of Rokhlin actions on UHF-algebras is much more subtle. For example, the main result of
\cite{Gar_automatic_2017} asserts that only totally disconnected groups can act on a UHF-algebra with the Rokhlin
property. Here, we only need to know that totally disconnected groups admit actions on UHF-algebras with the
(strong) Rokhlin property. We need some notation:

\begin{df}
Let $G$ be a totally disconnected group, and let $\mathcal{P}$ denote the set of all prime numbers.
Define a supernatural number $s_G\colon \mathcal{P}\to \{0,\I\}$ by $s_G(p)=\I$ if and only if
$p$ divides the order of some finite quotient of $G$. We denote by $D_G$ the UHF-algebra of infinite type
associated to $s_G$. Explicitly, $D_G=\bigotimes_{p\in\mathcal{P}} M_{p^{s_G(p)}}$.
\end{df}

\begin{thm}\label{thm:TotDiscUHF}
Let $G$ be a totally disconnected group. Then there exists an action $\delta_G\colon G\to\Aut(D_G)$
with the strong Rokhlin property.
\end{thm}
\begin{proof} See \cite{GarKalLup_totally_2017}. \end{proof}

The construction of $\delta_G$ is straightforward when $G$ is finite. The general case is, however, much more
complicated, and it involves constructing approximately representable coactions of $G$ on UHF-algebras, and using
noncommutative duality. The essential ideas already show up in the case when $G$ is abelian; see~\cite{GarLup_cocycle_2016}.

Combining these results, we obtain the following:

\begin{cor}\label{cor:generic} Let $A$ be a separable \uca\ and let $G$ be a compact
group.
\be\item If $A\otimes\Ot\cong A$, then the set of all $G$-actions on $A$ with the \Rp\ is a dense $G_\delta$-set in $\mbox{Act}_G(A)$.
\item If $G$ is totally disconnected and $A\otimes D_G\cong A$, then the set of all $G$-actions on $A$ with the \Rp\ is a dense
$G_\delta$-set in $\mbox{Act}_G(A)$.
\ee
\end{cor}
\begin{proof} (1). By \autoref{lma:OtStrRp}, there is an action $\gamma\colon G\to\Aut(\Ot)$ with the strong \Rp.
Since $A$ absorbs $\Ot$ tensorially,
the hypotheses of \autoref{Rokhlin are generic on D-absorbing algs} are met by Theorem~7.2.2 in \cite{RorSto_classification_2002},
and the result follows.

(2). This is analogous, using \autoref{thm:TotDiscUHF} instead of \autoref{lma:OtStrRp}. \end{proof}

\section{\texorpdfstring{$K$}{K}-theory and Cuntz semigroups of crossed products}

We begin this section by proving the main technical theorem that will be used in the
proofs of essentially every other result in this section. Roughly speaking, \autoref{thm:ApproxHomFixingFPA}
will allow us to take averages over the group $G$, in such a way that elements of the fixed point
algebra are left fixed, and also such that $\ast$-polynomial relations in the algebra are
approximately preserved.

We give a sketch of the proof in the case that $G$ is finite and there is an equivariant unital
embedding of $C(G)$ into the center of $A$, since this case is very simple and contains the main idea. 
Such an embedding gives us central projections $p_g\in A$,
for $g\in G$, which satisfy $\alpha_g(p_h)=p_{gh}$ for all $g,h\in G$ and $\sum\limits_{g\in G}p_g=1$.
Define a map $\psi\colon A\to A^\alpha$ by 
\[\psi(a)=\sum_{g\in G}p_g\alpha_g(a)p_g\]
for all $a\in A$. One can directly check that $\psi$ is a homomorphism that leaves $A^\alpha$ fixed, as desired.
In general, when the embedding of $C(G)$ is only almost central and almost multiplicative, one obtains an
almost homomorphism map.

\begin{thm}\label{thm:ApproxHomFixingFPA}
Let $A$ be a unital \ca, let $G$ be a second-countable compact group, and let
$\alpha\colon G\to\Aut(A)$ be an action with the \Rp. Given a compact subset $F\subseteq A$
and $\varepsilon>0$, there exists a unital, continuous linear map $\psi\colon A\to A^\alpha$
satisfying
\be
\item $\|\psi(ab)-\psi(a)\psi(b)\|<\ep$ for all $a,b\in F$;
\item $\|\psi(a^*)-\psi(a)^*\|<\ep$ for all $a\in F$;
\item $\|\psi(a)\|\leq 2\|a\|$ for all $a\in F$; and
\item $\psi(a)=a$ for all $a\in A^\alpha$.\ee

If $A$ is separable, it follows that there exists an approximate homomorphism
$(\psi_n)_{n\in\N}$ consisting of bounded, unital linear maps
$\psi_n\colon A\to A^\alpha$ satisfying $\psi_n(a)=a$
for all $a\in A^\alpha$. \end{thm}
\begin{proof} Without loss of generality, we may assume that $\|a\|\leq 1$ for all $a\in F$.
For the compact set $F$ and the tolerance $\ep_0=\frac{\ep}{6}$,
use Proposition~4.3 in \cite{Gar_regularity_2014} to find a positive number $\delta>0$, a
finite subset $K\subseteq G$, a partition of unity $(f_k)_{k\in K}$ in $C(G)$, and a
unital completely positive map $\varphi\colon C(G)\to A$, such that the following conditions hold:
\be\item[(a)] If $g$ and $g'$ in $G$ satisfy $d(g,g')<\delta$, then $\|\alpha_g(a)-\alpha_{g'}(a)\|<\ep_0$
for all $a\in F$.
\item[(b)] Whenever $k$ and $k'$ in $K$ satisfy $f_kf_{k'}\neq 0$, then $d(k,k')<\delta$.
\item[(c)] For every $g\in G$ and for every $a\in F\cup F^*$, we have
$$\left\| \ \alpha_g\left(\sum\limits_{k\in K} \varphi(f_{k})\alpha_{k}(a)\right)-
\sum\limits_{k\in K} \varphi(f_{k})\alpha_{k}(a)\right\|<\varepsilon_0.$$
\item[(d)] For every $a\in F\cup F^*$ and for every $k\in K$, we have
\[\left\|a\varphi(f_{k})-\varphi(f_{k})a\right\|<\frac{\ep_0}{|K|} \ \mbox{ and } \
\left\|a\varphi(f_{k})^{\frac{1}{2}}-\varphi(f_{k})^{\frac{1}{2}}a\right\|<\frac{\ep_0}{|K|}.\]
\item[(e)] Whenever $k$ and $k'$ in $K$ satisfy $f_{k}f_{k'}=0$, then
\[\left\|\varphi(f_k)\varphi(f_{k'})\right\|<\frac{\ep_0}{|K|}.\]
\ee

Denote by $E\colon A\to A^\alpha$ the canonical conditional expectation, which is given by
$E(a)=\int\limits_G\alpha_g(a)\ dg$ for all $a\in A$.
Define a unital linear map $\psi\colon A\to A^\alpha$ by
\[\psi(a)=E\left(\sum_{k\in K}\varphi(f_k)\alpha_k(a)\right)\]
for all $a\in A$. We claim that $\psi$ has the desired properties. It is immediate that
$\psi(a)=a$ for all $a\in A^\alpha$, using the properties of the conditional expectation $E$,
so condition (4) is guaranteed.

We proceed to check condition (1) in the statement.
Given $a,b\in F$, we use condition (c) at the second and fifth step,
conditions (a), (b), (d) and (e) at the third step, and the fact that $\varphi$ is unital and $(f_k)_{k\in K}$ is
a partition of unity of $C(G)$ at the fourth step, to get
\begin{align*}
\psi(a)\psi(b)&=E\left(\sum_{k\in K}\varphi(f_k)\alpha_k(a)\right)
E\left(\sum_{k'\in K}\varphi(f_{k'})\alpha_{k'}(b)\right)\\
&\approx_{2\ep_0}\sum_{k\in K}\sum_{k'\in K}\varphi(f_k)\alpha_k(a)
\varphi(f_{k'})\alpha_{k'}(b)\\
&\approx_{3\ep_0} \sum_{k\in K}\sum_{k'\in K}\varphi(f_k)\alpha_k(ab)
\varphi(f_{k'})\\
&= \sum_{k\in K}\varphi(f_k)\alpha_k(ab)\\
&\approx_{\ep_0}E\left(\sum_{k\in K}\varphi(f_k)\alpha_k(ab)\right)\\
&=\psi(ab).
\end{align*}
Hence $\|\psi(ab)-\psi(a)\psi(b)\|<6\ep_0=\ep$, and condition (1) is proved.

To prove condition (2), let $a\in F$. In the following computation, we use condition (d) at the second step, and
the fact that $\varphi$ and $E$ are positive at the third step:
\begin{align*} \psi(a^*)&=E\left(\sum_{k\in K}\varphi(f_k)\alpha_k(a)^*\right)\\
&\approx_{\ep_0} E\left(\sum_{k\in K}\alpha_k(a)^*\varphi(f_k)\right)\\
&= E\left(\sum_{k\in K}(\varphi(f_k)\alpha_k(a))^*\right)^*=\psi(a)^*.\end{align*}
Thus, condition (2) is verified.

To check condition (3), let $a\in F$. Then
\[ \psi(a)=E\left(\sum_{k\in K}\varphi(f_k)\alpha_k(a)\right)\\
\approx_{\ep_0} E\left(\sum_{k\in K}\varphi(f_k)^{\frac{1}{2}}\alpha_k(a)\varphi(f_k)^{\frac{1}{2}}\right).\]
Since the assignment $a\mapsto E\left(\sum\limits_{k\in K}\varphi(f_k)^{\frac{1}{2}}\alpha_k(a)\varphi(f_k)^{\frac{1}{2}}\right)$
defines a unital (completely) positive map $\widetilde{\psi}\colon A\to A^\alpha$, we conclude that
\[\|\psi(a)\|\leq \|\widetilde{\psi}(a)\|+\ep_0\leq \|a\|+\ep_0\leq 2\|a\|,\]
as desired.
This finishes the proof.\end{proof}

\subsection{K-theory and fixed point algebras.}
Our first applications of \autoref{thm:ApproxHomFixingFPA} are to the maps induced by the
canonical inclusion $A^\alpha \hookrightarrow A$ at the level of $K$-theory (\autoref{thm:InjK-Thy}),
and to the equivariant $K$-theory $K^G_\ast(A)$ (\autoref{thm:DiscrKthy}).

We start off with an intermediate result. Recall that for two projections $p$ and $q$ in a \ca\ $B$, we
say that $p$ is \emph{Murray-von Neumann subequivalent (in B)} to $q$, written $p\precsim_{\mathrm{M-vN}} q$,
if there exists a third projection $q_0\in B$ such that $p\sim_{\mathrm{M-vN}} q_0$ and $q_0\leq q$. Murray-von
Neumann subequivalence is easily seen to be transitive. (We warn the reader that $p\precsim_{\mathrm{M-vN}} q$
and $q\precsim_{\mathrm{M-vN}} p$ do not in general imply $p\sim_{\mathrm{M-vN}} q$.)

\begin{prop}\label{prop:MvNeq}
Let $A$ be a \uca, let $G$ be a second-countable compact group, and let $\alpha\colon G\to \Aut(A)$ be
an action with the Rokhlin property. Let $p,q\in A^\alpha$ be two projections.
\be\item Suppose that $p\precsim_{\mathrm{M-vN}} q$ in $A$. Then $p\precsim_{\mathrm{M-vN}} q$ in $A^\alpha$.
\item Suppose that $p\sim_{\mathrm{M-vN}} q$ in $A$. Then $p\sim_{\mathrm{M-vN}} q$ in $A^\alpha$.
\ee
\end{prop}
\begin{proof}
(1). Find a projection $q_0\in A$ such that $p\sim_{\mathrm{M-vN}} q_0$ and $q_0\leq q$. Find a partial isometry
$s\in A$ such that $s^*s=p$ and $ss^*=q_0$. For $\ep=\frac{1}{21}$, find $\delta_0>0$ such that
whenever $x\in A^\alpha$ satisfies $\|x^*x-x\|<\delta_0$, then there exists a projection $r\in A^\alpha$
such that $\|r-x\|<\ep$. Set $\delta=\min\left\{\frac{\dt_0}{3},\ep\right\}$.

Find a unital, continuous linear map $\psi\colon A\to A^\alpha$ as in the conclusion of
\autoref{thm:ApproxHomFixingFPA} for $\delta$ and $F=\{p,q,q_0,s,s^*\}$. Then
\begin{align*} \|\psi(q_0)^*\psi(q_0)-\psi(q_0)\|&\leq \|\psi(q_0)^*\psi(q_0)-\psi(q_0^*)\psi(q_0)\|+\|\psi(q_0)\psi(q_0)-\psi(q_0^2)\|\\
&< 2\dt+\dt<\dt_0.\end{align*}
By the choice of $\delta_0$, there exists a projection $\widetilde{q}_0$ in $A^\alpha$ such that
$\|\psi(q_0)-\widetilde{q}_0\|<\ep$. We claim that $\widetilde{q}_0\precsim_{\mathrm{M-vN}}q$ in $A^\alpha$

We use the identity $q_0=q_0q$ at the second step to show that
\[\|\widetilde{q}_0-\widetilde{q}_0q\|\leq 2\|\widetilde{q}_0-\psi(q_0)\|+\|\psi(q_0)-\psi(q_0)q\|
<2\ep+\dt\leq 1.\]
It follows from Lemma~2.5.2 in~\cite{Lin_introduction_2001} that $\widetilde{q}_0\precsim_{\mathrm{M-vN}}q$ in $A^\alpha$.

Now we claim that $p\sim_{\mathrm{M-vN}} \widetilde{q}_0$ in $A^\alpha$.
We use the identity $ps^*q_0s p=p$ at the last step to show
\begin{align*}
\|(\widetilde{q}_0\psi(s)p)^*(\widetilde{q}_0\psi(s)p)-p\|&\leq \|p\psi(s)^*\widetilde{q}_0^2\psi(s)p-p\psi(s^*)\widetilde{q}_0^2\psi(s)p\| + 4 \|\widetilde{q}_0-\psi(q_0)\|\\
& \ \ \ \ + \|\psi(p)\psi(s^*)\psi(q_0)\psi(s)\psi(p)-\psi(s^*s)\|\\
&< 2\dt+4\ep + 15\dt + \|\psi(ps^*q_0 s p)-\psi(s^*s)\|<21\ep=1.
\end{align*}
Likewise, $\|(\widetilde{q}_0\psi(s)p)(\widetilde{q}_0\psi(s)p)^*-q\|<1$.
By Lemma~2.5.3 in \cite{Lin_introduction_2001} applied to $\widetilde{q}_0\psi(s)p$, there exists a partial isometry $t$ in $A^\alpha$ such that
$t^*t=p$ and $tt^*=\widetilde{q}_0$.

We conclude that
\[p\sim_{\mathrm{M-vN}}\widetilde{q}_0 \precsim_{\mathrm{M-vN}} q\]
in $A^\alpha$, so the proof is complete.

(2). The proof of this part is analogous, so we omit the details.
\end{proof}

Part (1) of the following theorem generalizes Theorem~3.13 in~\cite{Izu_finiteI_2004} in two ways: we do not
assume our algebras to be simple, and we consider compact groups. It has also been independently observed in
\cite{BarSza_sequentially_2016}.
The remaining parts are new even when $G$ is finite.

\begin{thm}
\label{thm:InjK-Thy}
Let $A$ be a unital, separable \ca, let $G$ be a second-countable compact group, and let
$\alpha\colon G\to\Aut(A)$ be an action with the \Rp. Then the following assertions hold:
\be\item
The canonical inclusion $\iota\colon A^\alpha\to A$ induces an injective map
\[K_\ast(\iota)\colon K_\ast(A^\alpha)\to K_\ast(A).\]
\item The map $K_0(\iota)\colon K_0(A^\alpha)\to K_0(A)$ is an order-embedding; that is, whenever
$x,y\in K_0(A^\alpha)$ satisfy $K_0(\iota)(x)\leq K_0(\iota)(y)$ in $K_0(A)$, then $x\leq y$ in $K_0(A^\alpha)$.
\item Let $j\in \{0,1\}$, and let $H$ be a finitely generated subgroup of $K_j(A)$. Then there
exists a group homomorphism $\pi\colon H\to K_0(A^\alpha)$ such that
\[\pi\circ\iota|_{K_0(\iota)^{-1}(H)}=\id_{K_0(\iota)^{-1}(H)}.\]
\ee
\end{thm}
\begin{proof}
(1). The result for $K_1$ follows from the result for $K_0$, by tensoring with any unital
\ca\ satisfying the UCT with $K$-theory $(0,\Z)$, and using the K\"unneth formula. We
will therefore only prove the theorem for $K_0$.

Let $x\in K_0(A^\alpha)$ satisfy $K_0(\iota)(x)=0$.
Choose $n\in\N$ and projections $p,q\in M_n(A^\alpha)$ such that $x=[p]-[q]$. Then $[p]=[q]$ in $K_0(A)$.
Without loss of generality,
we may assume that $p$ and $q$ are Murray-von Neumann equivalent in $M_n(A)$.
Since the action $\alpha\otimes\id_{M_n}$ of $G$ on $M_n(A)$ has the Rokhlin property
by part (1) of \autoref{thm: permanence properties}, it follows from part (2) of \autoref{prop:MvNeq}
that $p$ and $q$ are Murray-von Neumann equivalent in $M_n(A^\alpha)$. Hence $x=0$ and $K_0(\iota)$ is
injective.

\vspace{0.3cm}

(2). Let $x,y\in K_0(A^\alpha)$ and suppose that $K_0(\iota)(x)\leq K_0(\iota)(y)$. Choose $n\in\N$ and
projections $p,q,e,$ and $f$ in $M_n(A^\alpha)$, such that $x=[p]-[q]$ and $y=[e]-[f]$. Then
$[p]+[f]\leq [e]+[q]$ in $K_0(A)$. Without loss of generality, we can assume that
$p\oplus f\precsim_{\mathrm{M-vN}} e\oplus q$ in $M_{2n}(A)$. Since the action induced by $\alpha$ on $M_{2n}(A)$
has the Rokhlin property by part (1) of \autoref{thm: permanence properties}, it follows from part (2) of
\autoref{prop:MvNeq} that $p\oplus f\precsim_{\mathrm{M-vN}} e\oplus q$ in $M_{2n}(A^\alpha)$.
Hence $[p]+[f]\leq [e]+[q]$ in $K_0(A^\alpha)$ and thus $x\leq y$ in $K_0(A^\alpha)$, as desired.

\vspace{0.3cm}

(3). We prove the statement only for $K_0$, which without loss of generality we assume
is not the trivial group. So let $H$ be a finitely generated subgroup of $K_0(A)$, and
choose an integer $m\geq 0$, and $k_0,\ldots,k_m\in \N$ such that $H$ is isomorphic to
$\Z^{k_0}\oplus\Z_{k_1}\oplus\cdots\oplus \Z_{k_m}$ as abelian groups. For $1\leq i\leq k_0$
and $1\leq j\leq m$, choose projections
\[p_{0,i}, q_{0,i},p_j,q_j\in M_\I(A),\]
such that $[p_{0,i}]- [q_{0,i}]$ is a (free) generator of the $i$-th copy of $\Z$; and
$[p_j]-[q_j]$ is a generator (of order $k_j$) of $\Z_{k_j}$. Without loss of generality, we
may assume that there are unitaries $u_j\in M_\I(A)$, for $j=1,\ldots,m$, such that
\[u_j \diag(\underbrace{p_j,\ldots,p_j}_{k_j\text{ times}}) u_j^*=\diag(\underbrace{q_j,\ldots,q_j}_{k_j\text{ times}}).\]

For $i=1,\ldots,k_0$, let $\Z^{(i)}$ denote the $i$-th copy of $\Z$ in $H$.
Then the intersection $\Z^{(i)}\cap K_0(\iota)(K_0(A^\alpha))$
is a subgroup of $\Z^{(i)}$, so there exists $n_{i,0}\geq 0$ such that
\[\Z^{(i)}\cap K_0(\iota)(K_0(A^\alpha))=n_i\Z^{(i)}.\]
Find projections $r_{i,0},s_{i,0}\in M_\I(A^\alpha)$ such that
\[[r_{i,0}]-[s_{i,0}]=n_i([p_{0,i}]-[q_{0,i}]) \in K_0(A).\]
Without loss of generality, we may assume that there exists a unitary $v_{i,0}\in M_\I(A)$ such that
\[v_{i,0} \diag(r_{i,0},\underbrace{q_{i,0},\ldots,q_{i,0}}_{n_i\text{ times}}) v_{i,0}^*=
\diag(s_{i,0},\underbrace{p_{i,0},\ldots,p_{i,0}}_{n_i\text{ times}}).\]

Likewise, for $j=1,\ldots,m$, there exists $\ell_j$ dividing $k_j$, such that
\[\Z_{k_j}\cap K_0(\iota)(K_0(A^\alpha))=\ell_j\Z_{k_j}.\]
Find projections $r_{j},s_{j}\in M_\I(A^\alpha)$ such that
\[[r_{j}]-[s_{j}]=\ell_j([p_{j}]-[q_{j}]) \in K_0(A).\]
Without loss of generality, we may assume that there exists a unitary $v_{j}\in M_\I(A)$ such that
\[v_{j} \diag(r_{j},\underbrace{q_j,\ldots,q_j}_{\ell_j\text{ times}}) v_{j}^*=
\diag(s_{j},\underbrace{p_j,\ldots,p_j}_{\ell_j\text{ times}}).\]

\emph{Claim 1:} the set
\[\{[r_{i,0}]-[s_{i,0}]\colon i=1,\ldots,k, n_i\neq 0\}\cup\{[r_{j}]-[s_{j}]\colon j=1,\ldots,m, \ell_j\neq 0\}\subseteq K_0(A^\alpha)\]
generates $K_0(A^\alpha)\cap K_0(\iota)^{-1}(H)$. By construction, the above set, when regarded as a subset of $K_0(A)$
under $K_0(\iota)$, generates $K_0(\iota)(K_0(A^\alpha))\cap H$. Moreover, since $K_0(\iota)$
is injective by part (1) of this theorem, the claim follows.

An analogous argument shows that $[r_{i,0}]-[s_{i,0}]$, for $i=1,\ldots,k_0$, has infinite order, and that
$[r_{j}]-[s_{j}]$, for $j=1,\ldots,m$, has order $\ell_j$.

Find $n$ large enough such that the set
\[F= \left\{p_{0,i}, q_{0,i},r_{0,i},s_{0,i},v_{0,i}\colon i=1,\ldots,k_0\right\}
\cup\left\{p_j,q_j,u_j,r_{j},s_{j},v_{j}\colon j=1,\ldots,m\right\}\]
is contained in $M_n(A)$.
Since the amplification of $\alpha$ to $M_n(A)$ has the Rokhlin property by part (1) of
\autoref{thm: permanence properties}, we can replace $M_n(A)$ with $A$.

Set $\ep=\frac{1}{10}$. Find $\delta_1>0$ such that whenever $x\in A$ satisfies $\|x^*x-x\|<\dt_1$, then
there exists a projection $p\in A$ with $\|p-x\|<\ep$. Find $\dt_2>0$ such that whenever $y\in A$
satisfies $\|y^*y-1\|<\dt_2$ and $\|yy^*-1\|<\dt_2$, then there exists a unitary $u\in A$ with
$\|u-y\|<\ep$.
Set $\dt=\min\left\{\frac{\ep}{3},\frac{\dt_1}{3},\frac{\dt_2}{3}\right\}$.
Use \autoref{thm:ApproxHomFixingFPA} to find a unital, linear map
$\psi\colon A\to A^\alpha$ satisfying
\be
\item[(a)] $\|\psi(a)\psi(b)-\psi(ab)\|<\dt$ for all $a,b\in F\cup F^2$;
\item[(b)] $\|\psi(a)^*-\psi(a^*)\|<\dt$ for all $a\in F\cup F^2$;
\item[(c)] $\|\psi(a)\|\leq 2\|a\|$ for all $a\in F\cup F^2$; and
\item[(d)] $\psi(a)=a$ for all $a\in A^\alpha$.\ee

Let $x\in \{p_{i,0},q_{i,0},p_{j},q_{j}\colon i=1,\ldots,k_0, j=1,\ldots,m\}$. Then
\begin{align*}
\left\|\psi(x)^*\psi(x)-\psi(x)\right\| &\leq
\left\|\psi(x)^*\psi(x)-\psi(x)\psi(x)\right\|
+\left\|\psi(x)\psi(x)-\psi(x^2)\right\|\\
&\leq 2\dt+\dt<\dt_1.
\end{align*}
Using the choice of $\delta_1$, find (and fix) a projection $\widetilde{x}\in A^\alpha$
such that $\|\widetilde{x}-\psi(x)\|<\ep$.

For $y\in \{u_j,v_j\colon j=1,\ldots,m\}$, we have
\begin{align*}
\left\|\psi(y)^*\psi(y)-1\right\| &\leq
\left\|\psi(y)^*\psi(y)-\psi(y^*)\psi(y)\right\|
+\left\|\psi(y^*)\psi(y)-1\right\|\\
&\leq 2\dt+\dt<\dt_2.
\end{align*}
Likewise, $\left\|\psi(y)\psi(y)^*-1\right\|$. Using the choice of $\dt_2$,
find (and fix) a unitary $\widetilde{y}\in A^\alpha$ such that
$\|\widetilde{y}-\psi(y)\|<\ep$.

\emph{Claim 2:} Let $j=1,\ldots,m$. Then $k_j\left([\widetilde{p}_{j}]-[\widetilde{q}_{j}]\right)=0$
in $K_0(A^\alpha)$. In the estimates below, there are $k_j$ repetitions on each diagonal matrix:
\begin{align*}
&\left\|\widetilde u_{j} \diag(\widetilde p_{j},\ldots,\widetilde p_{j}) \widetilde u_{j}^*-
\diag(\widetilde q_{j},\ldots,\widetilde q_{j}) \right\|\\
&\leq \left\|\widetilde u_{j} \diag(\widetilde p_{j},\ldots,\widetilde p_{j}) \widetilde u_{j}^*-
\psi(u_j) \diag(\psi(p_{j}),\ldots,\psi(p_{j})) \psi(u_{j})^*\right\|\\
& \ \ \ \ +\left\|\psi(u_j) \diag(\psi(p_{j}),\ldots,\psi(p_{j})) \psi(u_{j})^*-
\diag(\psi(q_{j}),\ldots,\psi(q_{j})) \right\| \\
& \ \ \ \ + \left\|\diag(\psi(q_{j}),\ldots,\psi(q_{j})) -\diag(\widetilde q_{j},\ldots,\widetilde q_{j}) \right\|\\
&\leq 7\ep+3\dt+\ep \leq 9\ep<1.
\end{align*}
It follows from Lemma~2.5.3 in \cite{Lin_introduction_2001} that $\widetilde u_{j} \diag(\widetilde p_{j},\ldots,\widetilde p_{j}) \widetilde u_{j}^*$ is Murray-von Neumann equivalent to
$\diag(\widetilde q_{j},\ldots,\widetilde q_{j})$ in $A^\alpha$, and the claim is proved.

Define a map $\pi\colon H\to K_0(A^\alpha)$ by
\[\pi([p_{i,0}]-[q_{i,0}])=[\widetilde{p}_{i,0}]-[\widetilde{q}_{i,0}],\]
for $i=1,\ldots,k_0$, and
\[\pi([p_{j}]-[q_{j}])=[\widetilde{p}_{j}]-[\widetilde{q}_{j}],\]
for $j=1,\ldots,m$. Then $\pi$ is a well-defined group homomorphism by the previous claim.

\emph{Claim 3:} We have $\pi\circ K_0(\iota)|_{K_0(\iota)^{-1}(H)}=\id_{K_0(\iota)^{-1}(H)}$.
For this, it is enough to check that
\[\pi\left(K_0(\iota)([r_{i,0}]-[s_{i,0}])\right)=[r_{i,0}]-[s_{i,0}]\]
for all $i=1,\ldots,k_0$, and
\[\pi\left(K_0(\iota)([r_{j}]-[s_{j}])\right)=[r_{j}]-[s_{j}]\]
for all $j=1,\ldots,m$.

Fix $i\in \{1,\ldots,k_0\}$. Then $K_0(\iota)([r_{i,0}]-[s_{i,0}])=n_i([p_{i,0}]-[q_{i,0}])$ in $K_0(A)$, so
we shall prove that
\[n_i([\widetilde{p}_{i,0}]-[\widetilde{q}_{i,0}])=[r_{i,0}]-[s_{i,0}]\]
in $K_0(A^\alpha)$. The following estimate can be shown in a way similar to what was done in the proof
of Claim 2:
\[\left\|\widetilde v_{j} \diag(s_{i,0},\underbrace{\widetilde{p}_{i,0},\ldots,\widetilde{p}_{i,0}}_{n_i\text{ times}}) \widetilde v_{j}^*-
\diag(\widetilde r_{i,0},\underbrace{\widetilde{q}_{i,0},\ldots,\widetilde{q}_{i,0}}_{n_i\text{ times}}) \right\|<1.
\]

Again, it follows from Lemma~2.5.3 in \cite{Lin_introduction_2001} that
$\widetilde v_{j} \diag(s_{i,0},\widetilde p_{i,0},\ldots,\widetilde p_{i,0}) \widetilde v_{j}^*$ is Mur\-ray-von Neumann equivalent to
$\diag(\widetilde r_{i,0},q_{i,0},\ldots,\widetilde q_{i,0})$ in $A^\alpha$, and the claim is proved.
The argument for $[r_j]-[s_j]$, for $j=1,\ldots,m$, is analogous.
\end{proof}

It follows that certain features of the $K$-groups of $A$ are inherited by the $K$-groups
of $A^\alpha$ and $A\rtimes_\alpha G$:

\begin{cor}\label{cor:K-thyAtoAG}
Let $A$ be a unital \ca, let $G$ be a second-countable compact group, and let
$\alpha\colon G\to\Aut(A)$ be an action with the \Rp. Let $j\in \{0,1\}$ and suppose
that $K_j(A)$ is either:
\be\item free;
\item torsion;
\item torsion-free;
\item finitely generated;
\item zero.\ee
Then the same holds for $K_j(A^\alpha)$ and $K_j(A\rtimes_\alpha G)$.
\end{cor}
\begin{proof}
This follows immediately from part~(1) of~\autoref{thm:InjK-Thy}.
\end{proof}

\begin{cor}
Let $A$ be a unital, separable \ca, let $G$ be a second-countable compact group, and let
$\alpha\colon G\to\Aut(A)$ be an action with the \Rp. Suppose that $K_j(A)$ is finitely generated
for some $j\in \{0,1\}$. Then $K_j(A^\alpha)$ is isomorphic to a direct summand in $K_j(A)$.
\end{cor}
\begin{proof}
This is consequence of part (3) of \autoref{thm:InjK-Thy}, together with the fact that a short
exact sequence $0\to G_1\to G_2\to G_3\to 0$ of abelian groups with a section $G_2\to G_1$ must
split.
\end{proof}

In the case of Rokhlin actions of finite groups on simple \ca s, Izumi showed in Theorem~3.13 in~\cite{Izu_finiteI_2004} that the 
image of $K_0(\iota)\colon K_0(A^\alpha)\to K_(A)$ can be computed quite explicitly as follows: 
\[K_\ast(\iota)(K_\ast(A^\alpha))=\{x\in K_\ast(A)\colon K_\ast(\alpha_g)(x)=x \mbox{ for all } g\in G\}.\]

Although $K_\ast(\iota)$ is also injective for Rokhlin actions of arbitrary compact groups (part~(1) of~\autoref{thm:InjK-Thy}),
a formula analogous to the one above is not true in general:

\begin{eg}
Let $A$ be a unital Kirchberg algebra satisfying the UCT, with $K$-theory given by $K_0(A)\cong K_1(A)\cong\Z_6$,
with the class of the unit in $K_0(A)$ corresponding to $3\in \Z_6$. By
Theorem~5.3 in \cite{Gar_classificationII_2014}, there exists a circle action $\alpha\colon \T\to\Aut(A)$ with the Rokhlin
property, such that $K_0(A^\alpha)\cong \Z_2$, with the class of the unit of $A^\alpha$ corresponding to $1\in \Z_2$,
and $K_1(A^\alpha)\cong \Z_3$.

By Proposition~3.9 in \cite{Gar_classificationI_2014}, if $\zeta\in\T$ and $j\in \{0,1\}$,
then $K_j(\alpha_\zeta)=\id_{K_j(A)}$. In particular,
\[K_j(\iota)(K_j(A^\alpha))\cong \Z_{j+2}\ncong \Z_6\cong \{x\in K_j(A)\colon K_j(\alpha_\zeta)(x)=x \mbox{ for all } \zeta\in \T\}.\]
\end{eg}

What goes wrong in the example above (the notation of which we keep), is that if $p$ is a projection in $A$ such that
$\alpha_\zeta(p)$ is unitarily equivalent to $p$ for all $\zeta\in \T$, then the unitaries $u_\zeta$, for
$\zeta\in\T$, which
implement the unitary equivalence, cannot in general be chosen to depend continuously on $\zeta$.

The splitting constructed in part (3) of \autoref{thm:InjK-Thy} is not natural with respect to equivariant
homomorphisms between $G$-algebras with the Rokhlin property. On the other hand, the splitting \emph{can}
be shown to be natural with respect to certain maps; see \autoref{thm:NaturalitySplitting}. The main
ingredient is the following relative version of \autoref{thm:ApproxHomFixingFPA}. Since its assumptions
are long and are needed also in \autoref{thm:NaturalitySplitting}, we describe them separately. 
\ \\

\textbf{General setting:} 
Let $B$ be \uca\ and let $A$ be a unital subalgebra of $B$ (with the same unit), and let $\mu\colon A\to B$
be its canonical inclusion. Let $G$ be a compact group, and let
$\alpha\colon G\to\Aut(A)$ and $\beta\colon G\to\Aut(B)$ be actions with the Rokhlin property.
Let $\tau\colon G\to G$ be a surjective group homomorphism satisfying
\[\beta_{g}(a)=\alpha_{\tau(g)}(a)\]
for all $g\in G$ and all $a\in A$. (Note, in particular, that $\mu$ restricts to the inclusion
$A^\alpha\subseteq B^\beta$.) Write $\iota_A\colon A^\alpha\to A$ and $\iota_B\colon B^\beta\to B$
for the canonical inclusions.

Suppose that there are unital, equivariant homomorphisms
\[\varphi_A\colon C(G)\to A_{\I,\alpha}\cap B'\subseteq A_{\I,\alpha}\cap A' \
\mbox{ and } \ \varphi_B\colon C(G)\to B_{\I,\beta}\cap B'\]
making the following diagram commute:
\beqa\xymatrix{
C(G)\ar[r]^-{\tau^*}\ar[d]_-{\varphi_A} &C(G)\ar[d]^-{\varphi_A}\\
A_{\I,\alpha}\cap B'\ar@{^{(}->}[r] &B_{\I,\beta}\cap B'.
}\eeqa

(A typical situation in which the above assumptions are met, is as follows. Let $\beta\colon \T\to\Aut(B)$
be a Rokhlin action, and let $\tau\colon \T\to\T$ be the square map. Let $A$ be the subalgebra of $B$
of elements that are fixed by $\{1,-1\}\subseteq \T$, and let $\alpha\colon G\to\Aut(A)$ be the induced
action. It can be seen that $\alpha$ has the Rokhlin property, and that there exist maps $\varphi_A$ and 
$\varphi_B$ as above; this is done with detail in \cite{Gar_automatic_2017}.) 

\begin{prop}\label{prop:relative}
Adopt the assumptions and notation in the general setting described above.
Let $F_A\subseteq A$ and $F_B\subseteq B$ be compact subsets with $F_A\subseteq F_B$, and let $\ep>0$.
Then there exist unital, continuous, linear maps $\psi_A\colon A\to A^\alpha$ and $\psi_B\colon B\to B^\beta$,
satisfying the following conditions:
\be
\item $\|\psi_B(bc)-\psi(b)\psi(c)\|<\ep$ for all $b,c\in F_B$;
\item $\|\psi_B(b^*)-\psi_B(b)^*\|<\ep$ for all $b\in F_B$;
\item $\|\psi_B(b)\|\leq 2\|b\|$ for all $b\in F_B$; and
\item $\psi_B(b)=b$ for all $b\in B^\beta$.
\item $\psi_B|_A=\psi_A$.
\ee
\end{prop}
\begin{proof}
Using Choi-Effros lifting theorem for the map $\varphi_A$ as in
Proposition~4.3 in \cite{Gar_regularity_2014}, find a positive number $\delta>0$, a
finite subset $K\subseteq G$, a partition of unity $(f_k)_{k\in K}$ in $C(G)$, and a
unital completely positive map $\varphi\colon C(G)\to A$, such that the following conditions hold:
\be\item[(a)] If $g,g'\in G$ satisfy $d(g,g')<\delta$, then $\|\alpha_g(a)-\alpha_{g'}(a)\|<\ep_0$
for all $a\in F_A$ and $\|\beta_g(b)-\beta_{g'}(b)\|<\ep_0$
for all $b\in F_B$.
\item[(b)] Whenever $k$ and $k'$ in $K$ satisfy $f_kf_{k'}\neq 0$, then $d(k,k')<\delta$.
\item[(c)] For every $g\in G$ and for every $a\in F_A\cup F_A^*$, we have
\[\left\| \ \alpha_g\left(\sum\limits_{k\in K} \varphi(f_{k})\alpha_{\tau(k)}(a)\right)-
\sum\limits_{k\in K} \varphi(f_{k})\alpha_{\tau(k)}(a)\right\|<\varepsilon_0.\]
\item[(c')] For every $g\in G$ and for every $b\in F_B\cup F_B^*$, we have
\[\left\| \ \beta_g\left(\sum\limits_{k\in K} \varphi(f_{k})\beta_{k}(b)\right)-
\sum\limits_{k\in K} \varphi(f_{k})\beta_{k}(b)\right\|<\varepsilon_0.\]
\item[(d)] For every $a\in F_A\cup F_A^*$ and for every $k\in K$, we have
\[\left\|a\varphi(f_{k})-\varphi(f_{k})a\right\|<\frac{\ep_0}{|K|} \ \mbox{ and } \
\left\|a\varphi(f_{k})^{\frac{1}{2}}-\varphi(f_{k})^{\frac{1}{2}}a\right\|<\frac{\ep_0}{|K|}.\]
\item[(d')] For every $b\in F_B\cup F_B^*$ and for every $k\in K$, we have
\[\left\|b\varphi(f_{k})-\varphi(f_{k})b\right\|<\frac{\ep_0}{|K|} \ \mbox{ and } \
\left\|b\varphi(f_{k})^{\frac{1}{2}}-\varphi(f_{k})^{\frac{1}{2}}b\right\|<\frac{\ep_0}{|K|}.\]
\item[(e)] Whenever $k$ and $k'$ in $K$ satisfy $f_{k}f_{k'}=0$, then
\[\left\|\varphi(f_k)\varphi(f_{k'})\right\|<\frac{\ep_0}{|K|}.\]
\ee

Define unital linear maps $\psi_A\colon A\to A^\alpha$ and $\psi_B\colon B\to B^\beta$ by
\[\psi_A(a)=\int_G\alpha_g\left(\sum_{k\in K}\varphi(f_k)\alpha_{\tau(k)}(a)\right)dg \ \mbox{ and } \
\psi_B(b)=\int_G\beta_g\left(\sum_{k\in K}\varphi(f_k)\beta_k(b)\right)dg\]
for all $a\in A$ and all $b\in B$. Properties (4) and (5) in the statement are immediate, while
the remaining ones are checked as in the proof of \autoref{thm:ApproxHomFixingFPA}. We omit the
details.
\end{proof}

Naturality of the splitting in part~(3) of \autoref{thm:InjK-Thy}, for homomorphisms as in
the proposition above, follows immediately by using \autoref{prop:relative} in the proof of
\autoref{thm:InjK-Thy}. Despite the seemingly restrictive assumptions, this version of naturality
turns out to be crucial in the proof of the main result in \cite{Gar_automatic_2017}.

\begin{thm}\label{thm:NaturalitySplitting}
Adopt the assumptions and notation in the general setting described before \autoref{prop:relative}.
Let $j\in \{0,1\}$ and let $H$ be a finitely generated subgroup of $K_j(A)$. Set $\widetilde{H}=K_j(\mu)(H)$,
which is a finitely generated subgroup of $K_j(B)$. Then there exist group homomorphisms
$\pi_A\colon H\to K_0(A^\alpha)$ and $\pi_B\colon \widetilde{H}\to K_0(B^\beta)$, satisfying
\[\pi_A\circ K_j(\iota_A)|_{K_j(\iota_A)^{-1}(H)}=\id_{K_j(\iota_A)^{-1}(H)}\]
and
\[\pi_B\circ K_j(\iota_B)|_{K_j(\iota_B)^{-1}(\widetilde{H})}=\id_{K_j(\iota_B)^{-1}(\widetilde{H})},\]
which moreover make the following diagram commute:
\beqa\xymatrix{
K_j(A^\alpha)\ar@/^1.75pc/[rr]^-{K_j(\iota_A)}\ar[d]_-{K_j(\mu)} & H \ar[d]^-{K_j(\mu)}\ar[l]_-{\pi_A}\ar@{^{(}->}[r]& K_j(A)\ar[d]^-{K_j(\mu)}\\
K_j(B^\beta)\ar@/_1.75pc/[rr]_-{K_j(\iota_B)} & \widetilde{H}\ar[l]^-{\pi_B}\ar@{^{(}->}[r]&K_j(B).}\eeqa
\end{thm}

We record the following result for later use.
The conclusion in part~(2) explains why it was possible in~\cite{Gar_classificationI_2014}
to classify circle actions with the Rokhlin property on Kirchberg algebras up to conjugacy, using
a cocycle conjugacy invariant (namely the equivariant $K$-theory).
In its proof, we will need to use the 1-cohomology set $H^1_\alpha(G,\U(A))$ of an action
$\alpha\colon G\to\Aut(A)$ of a locally group $G$ on a \uca\ $A$; see, for example, Subsection~2.1
in~\cite{Izu_finiteI_2004}.

\begin{prop}
Let $A$ be a separable \uca, let $G$ be a second-countable compact group, let
$\alpha,\beta\colon G\to\Aut(A)$ be actions. Assume that $\alpha$ has the Rokhlin property,
and that $\alpha$ and $\beta$ are exterior conjugate.
\be\item Then $\beta$ has the Rokhlin property.
\item If $A$ is moreover simple and purely infinite, then there exists a unitary $w\in \U(A)$ such that
$\Ad(w)\circ\alpha_g\circ\Ad(w^*)=\beta_g$ for all $g\in G$. That is, $\alpha$ and $\beta$ are conjugate
via an inner automorphism of $A$.
\ee
\end{prop}
\begin{proof}
(1). It is immediate to check that if two actions are exterior conjugate, then they induce
identical actions on the (continuous part of the) central sequence algebra. (Essentially because
the unitaries from the cocycle, which belong to $A$, commute with everything in the central sequence
algebra.) Thus, one of them has
the Rokhlin property if and only if the other ones does.

(2). We claim that $H^1_\alpha(G,\U(A))=\{0\}$.
By Theorem~4.14 in \cite{Gar_rokhlin_2017}, the action $\alpha_g$ is outer for all $g\in G\setminus\{1\}$,
so in particular $\alpha$ is faithful.
By Proposition~3.19 in \cite{Gar_crossed_2014}, the crossed product $A\rtimes_\alpha G$ is purely
infinite, and it is simple by Corollary~2.14 in~\cite{Gar_crossed_2014}.
Finally, $K_0(\iota)$ is injective by part (1) of \autoref{thm:InjK-Thy},
so the result follows from Proposition~2.5 in \cite{Izu_finiteI_2004}.

Now, suppose that $\alpha$ and $\beta$ are exterior conjugate, and let $u\colon G\to\U(A)$ be a
1-cocycle for $\alpha$ satisfying $\beta_g=\Ad(u_g)\circ \alpha_g$ for all $g\in G$. By the claim, there
exists a unitary $w\in \U(A)$ such that $u_g=w\alpha_g(w^*)$ for all $g\in G$. It is then immediate that
\[\Ad(w)\circ\alpha_g\circ \Ad(w^*)= \Ad(u_g)\circ\alpha_g=\beta_g\]
for all $g\in G$, proving the result.
\end{proof}

We proceed to briefly recall the definition of equivariant $K$-theory.

\begin{df} (See Chapter~2 in~\cite{Phi_equivariant_1987}).
Let $G$ be a compact group, let $A$ be a unital \ca\ and let $\alpha\colon G\to\Aut(A)$
be a continuous action. Denote by $\mathcal{P}_G(A)$ the set of all $G$-invariant projections
in all of the algebras of the form $\B(V)\otimes A$, for all \fd\ representations $v\colon G\to \U(V)$
(we take the diagonal action of $G$ on $\B(V)\otimes A$). Two $G$-invariant
projections $p$ and $q$ in $\mathcal{P}_G(A)$ are said to be
\emph{equivariantly Murray-von Neumann equivalent} if there exists a $G$-invariant partial isometry
$s$ in $\B(V,W)\otimes A$
such that $s^*s=p$ and $ss^*=q$. We let $V_G(A)$ denote the set
of equivalence classes in $\mathcal{P}_G(A)$ with addition given
by direct sum.

We define the \emph{equivariant $K_0$-group} of $A$, denoted $K_0^G(A)$, to be the Grothendieck group of
$V_G(A)$, and define the \emph{equivariant $K_1$-group} of $A$, denoted $K_1^G(A)$, to be $K_0^G(SA)$,
where the action of $G$ on $SA$ is trivial in the suspension direction.\end{df}

Recall that the representation ring $R(G)$ of a compact group $G$ is the Gro\-then\-dieck group of the
semigroup of unitary equivalence classes of finite dimensional (unitary) representations of $G$. Its concrete
picture is given by
\begin{align*}
R(G)=\left\{
[(V,v)]-[(W,w)]\colon
\begin{aligned}
&v\colon G\to \U(V) \mbox{ and } w\colon G\to\U(W) \mbox{ are } \\
& \ \mbox{ finite dimensional representations}
\end{aligned}
\right\}.
\end{align*}

\begin{rem} The equivariant $K$-theory of $A$ is a module over the representation ring $R(G)$ of
$G$, which can be identified with $K_0^G(\C)$, with the
operation given by tensor product. The induced operation $R(G)\times K_\ast^G(A)\to K_\ast^G(A)$
makes $K_\ast^G(A)$ into an $R(G)$-module. \end{rem}

We denote by $I(G)$ the augmentation ideal in $R(G)$, this is, the kernel of the
ring homomorphism $\dim\colon R(G)\to\Z$ given by
\[\dim([(V,v)]-[(W,w)])=\dim(V)-\dim(W)\]
for $[(V,v)]-[(W,w)]\in R(G)$.
The following is Definition~4.1.2 in \cite{Phi_equivariant_1987}.

\begin{df}\label{df:DiscrKT}
Let $G$ be a compact group, let $A$ be a unital \ca\ and let $\alpha\colon G\to\Aut(A)$
be a continuous action. We say that $\alpha$ has \emph{discrete $K$-theory} if there exists
$n\in\N$ such that $I(G)^n\cdot K_\ast^G(A)=\{0\}$.
\end{df}

In the following theorem, we show that the Rokhlin property implies discrete $K$-theory in a strong sense.

\begin{thm} \label{thm:DiscrKthy}
Let $G$ be a compact group, let $A$ be a \uca, and let $\alpha\colon G\to\Aut(A)$ be an action
with the Rokhlin property. Then $I(G)\cdot K^G_\ast(A)=0$. In particular, $\alpha$ has discrete
$K$-theory.
\end{thm}
\begin{proof}
We show the result for $K_0^G$; the result for $K_1^G$ is analogous. (It also follows by replacing
$A$ with $A\otimes B$, where $B$ is any \uca\ satisfying the UCT with $K_0(B)=\{0\}$ and $K_1(B)=\Z$,
and endowing it with the $G$-action $\alpha\otimes\id_B$.)

We show first that $I(G)\cdot V_G(A)=0$. Let $v\colon G\to \U(V)$ be a finite dimensional
representation, and let $p$ be a $G$-invariant projection in $\B(V)\otimes A$. Since the action
$g\mapsto \Ad(w) \otimes \alpha_g$ of $G$ on $\B(V)\otimes A$ has the Rokhlin property by part (1)
of \autoref{thm: permanence properties}, we may assume that $V=\C$, so that $p$ is a $G$-invariant
projection in $A$.

Let $x=[(W_1,w_1)]-[(W_2,w_2)]\in I(G)$ be given. Since $W_1\cong W_2$ as vector spaces, it is clear
that
\[x=\left([(W_1,w_1)]-[(W_1,\id_{W_1})]\right) - \left([(W_2,w_2)]-[(W_2,\id_{W_2})]\right).\]
In particular, it is enough to show that if $w\colon G\to\U(W)$ is a finite dimensional representation,
then $\left([(W_1,w_1)]-[(W_1,\id_{W_1})]\right)[p]=0$ in $K^G_0(A)$. We identify $M_2(\B(W))\otimes A$
with $M_2(\B(W,A))$ in the usual way. We will show that the elements
\[
\left(
  \begin{array}{cc}
    \id_W & 0 \\
    0 & 0 \\
  \end{array}
\right)\otimes p =\left(
  \begin{array}{cc}
    p & 0 \\
    0 & 0 \\
  \end{array}
\right)\in M_2(\B(W,A))\] and \[\left(
  \begin{array}{cc}
    0 & 0 \\
    0 & \id_W \\
  \end{array}
\right)\otimes p=\left(
  \begin{array}{cc}
    0 & 0 \\
    0 & p \\
  \end{array}
\right)\in M_2(\B(W,A))
\]
are Murray-von Neumann equivalent in the fixed point algebra of $M_2(\B(W,A))$. The action $\beta\colon G\to \Aut(M_2(\B(W))\otimes A)$ is given by
\[
\beta_g=\Ad\left(
               \begin{array}{cc}
                 w_g & 0 \\
                 0 & 1 \\
               \end{array}
             \right)\otimes \alpha_g,\]
which again has the Rokhlin property.

Let $0<\ep<\frac{1}{3}$, and find $\delta_0>0$ such that whenever $B$ is a \ca\ and $s\in B$ satisfies $\|s^*s-1\|<\delta_0$
and $\|ss^*-1\|<\delta$, then there exists a unitary $u$ in $B$ such that $\|u-s\|<\ep$.
Set $\delta=\min\{\delta_0,\frac{1}{5}\}$.

Set
$r= \left(
               \begin{array}{cc}
                 0 & 1 \\
                 1 & 0 \\
               \end{array}
             \right)\otimes 1_A$,
and observe that
\[\beta_g(r)= \left(
               \begin{array}{cc}
                 0 & w_g^* \\
                 w_g & 0 \\
               \end{array}
             \right)\otimes 1_A\]
for all $g\in G$. Set $F=\left\{r,r^*,\left(
                                   \begin{array}{cc}
                                     p & 0 \\
                                     0 & 0 \\
                                   \end{array}
                                 \right),\left(
                                   \begin{array}{cc}
                                     0 & 0 \\
                                     0 & p \\
                                   \end{array}
                                 \right)\right\}$.
Use \autoref{thm:ApproxHomFixingFPA} to find a continuous, unital linear map
\[\psi\colon M_2(\B(W))\otimes A\to \left(M_2(\B(W))\otimes A\right)^\beta\]
which is the identity on $\left(M_2(\B(W))\otimes A\right)^\beta$, and moreover satisfies
\[\|\psi(ab)-\psi(a)\psi(b)\|<\frac{\delta}{3} \ \mbox{ and } \ \|\psi(a^*)-\psi(a)^*\|<\frac{\delta}{2}\]
for all $a,b\in F$.

We have
\begin{align*}
\|\psi(r)^*\psi(r)-1\|&\leq \|\psi(r)^*\psi(r)-\psi(r^*)\psi(r)\|+\|\psi(r^*)\psi(r)-\psi(r^*r)\|\\
&\leq
2\|\psi(r)^*-\psi(r^*)\|+\frac{\delta}{3}<\delta\leq\delta_0.\end{align*}
Likewise, $\|\psi(r)\psi(r)^*-1\|< \delta_0$. By the choice of $\delta_0$, there exists a unitary
$u\in M_2(\B(W,A))^\beta$ such that $\|u-\psi(r)\|<\ep$.

Recall that $\psi\left(\left(
          \begin{array}{cc}
            p & 0 \\
            0 & 0 \\
          \end{array}
        \right)\right)=\left(
          \begin{array}{cc}
            p & 0 \\
            0 & 0 \\
          \end{array}
        \right)$ and similarly for $\left(
          \begin{array}{cc}
            0 & 0 \\
            0 & p \\
          \end{array}
        \right)$. Using this at the third step, we compute
\begin{align*}
\left\|u\left(
          \begin{array}{cc}
            p & 0 \\
            0 & 0 \\
          \end{array}
        \right)u^*-\left(
          \begin{array}{cc}
            0 & 0 \\
            0 & p \\
          \end{array}
        \right)\right\|&\leq 2\|u-\psi(r)\|+ \left\|\psi(r)\left(\begin{array}{cc}
            p & 0 \\
            0 & 0 \\
          \end{array}
        \right)u^*-\left(
          \begin{array}{cc}
            0 & 0 \\
            0 & p \\
          \end{array}
        \right)\right\|\\
        &\leq 2\ep+ \left\|\psi(r)\left(\begin{array}{cc}
            p & 0 \\
            0 & 0 \\
          \end{array}
        \right)u^*-\left(
          \begin{array}{cc}
            0 & 0 \\
            0 & p \\
          \end{array}
        \right)\right\|\\
&\leq 2\ep + 5\frac{\delta}{3}<1.
\end{align*}

By Lemma~2.5.3 in \cite{Lin_introduction_2001}, the projections $\left(
                                                      \begin{array}{cc}
                                                        p & 0 \\
                                                        0 & 0 \\
                                                      \end{array}
                                                    \right)
$ and $\left(
                                                      \begin{array}{cc}
                                                        0 & 0 \\
                                                        0 & p \\
                                                      \end{array}
                                                    \right)$ are Murray-von Neumann equivalent in
$M_2(\B(W,A))^\beta$.

Since $K_0^G(A)$ is generated by the image of $V_G(A)$, the result follows.
\end{proof}

Suppose that $G$ is abelian, and set $\Gamma=\widehat{G}$.
It follows from \autoref{thm:DiscrKthy} and the canonical $R(G)$-module
identification $K_\ast^G(A)\cong K_\ast(A\rtimes_\alpha G)$ given by Julg's Theorem
(see Theorem~2.6.1 in~\cite{Phi_equivariant_1987}
for the identification as abelian groups, and see Proposition~2.7.10 in~\cite{Phi_equivariant_1987}
for the identification as $R(G)\cong \Z[\Gamma]$-modules),
that $K_\ast(\widehat{\alpha}_\gamma)=\id_{K_\ast(A\rtimes_\alpha G)}$ for all $\gamma\in\Gamma$.

\subsection{Cuntz semigroup and fixed point algebras.}
In this section, we study the map induced by the inclusion $A^\alpha\to A$ at the level of the
Cuntz semigroup (\autoref{thm:InjCu}). In \autoref{cor:K-thyCuCP}, we relate the $K$-theory
and Cuntz semigroup of the crossed product to those of the original algebra.

We need to briefly review the definition of the Cuntz semigroup.
The material to be recalled, and much more, can be found in \cite{CowEllIva_cuntz_2008} and
\cite{AntPerThi_tensor_2014}

\begin{df} Let $A$ be a C*-algebra and let $a$ and $b$ be positive elements in $A$. We say that $a$
is \emph{Cuntz subequivalent} to $b$, and denote this by $a\precsim b$, if there is a
sequence $(x_n)_{n\in\N}$ in $A$ such that
\[\lim_{n\to\I}\left\|x_n^*bx_n- a\right\|=0.\]

We say that $a$ is \emph{Cuntz equivalent} to $b$, and denote this by $a\sim b$, if $a\precsim b$
and $b\precsim a$.\end{df}

It can be shown that $\precsim$ is a preorder on the set of positive elements of $A$, from which it
follows that $\sim$ is an equivalence relation. We denote by $[a]$ the Cuntz equivalence class of
the positive element $a\in A$.

\begin{df}
The Cuntz semigroup of a \ca\ $A$, denoted by $\Cu(A)$, is defined to be the set of Cuntz
equivalence classes of positive elements of $A\otimes \K$. Addition in $\Cu(A)$ is given by direct sum.
Moreover, $\Cu(A)$ becomes an ordered semigroup when equipped with the order given
by $[a]\leq [b]$ if $a\precsim b$, for $a,b\in (A\otimes\K)_+$.

If $A$ and $B$ are \ca s, any homomorphism $\varphi\colon A\to B$ induces an
order-preserving map $\Cu(\varphi)\colon \Cu(A)\to \Cu(B)$, given by
\[\Cu(\varphi)([a])=[(\varphi\otimes \id_\K)(a)]\]
for every positive element $a\in A\otimes\K$.\end{df}

Recall that a semigroup homomorphism $\sigma\colon S\to T$ between partially ordered semigroups $S$ and $T$
is said to be an \emph{order-embedding}, if $x,y\in S$ and $\sigma(x)\leq \sigma(y)$ in $T$,
imply $x\leq y$ in $S$. (Note that order-embeddings are in particular injective.)

\begin{thm}\label{thm:InjCu}
Let $A$ be a unital, separable \ca, let $G$ be a second-countable compact group, and let
$\alpha\colon G\to\Aut(A)$ be an action with the \Rp.
Then the canonical inclusion $\iota\colon A^\alpha\to A$
induces an order-embedding $\Cu(\iota)\colon \Cu(A^\alpha)\to \Cu(A)$.
\end{thm}
\begin{proof}
Let $x$ and $y$ be positive elements in $A^\alpha \otimes \K$ such that $x\precsim y$ in $A\otimes\K$.
By R\o rdam's Lemma (Proposition~2.4 in~\cite{Ror_structureI_1991}), given $\varepsilon>0$
there exist $k$ in $\N$, a positive number $\delta>0$ and $s$ in $A\otimes M_k$ such that $(x-\ep)_+=s^*s$
and $ss^*$ belongs to the hereditary subalgebra of $A\otimes M_k$ generated by $(y-\delta)_+$.
Note that the action $\alpha\otimes\id_{M_k}$ of $G$ on $A\otimes M_k$ has the \Rp\ by part (1) of
\autoref{thm: permanence properties}, and that $M_k(A)^{\alpha\otimes\id_{M_k}}$ can be canonically
identified with $M_k(A^\alpha)$. Let $(\psi_n)_{n\in\N}$ be a sequence of unital
completely positive maps $A\to A^\alpha$ as in the conclusion of \autoref{thm:ApproxHomFixingFPA}. For $n\in\N$,
we denote by $\psi^{(k)}_n\colon M_k(A)\to M_k(A^\alpha)$ the tensor product of $\psi_n$ with $\id_{M_k}$.
Since $s^*s=(x-\ep)_+$, we have
$$\lim_{n\to\I}\left\|\psi^{(k)}_n(s)^*\psi^{(k)}_n(s)-(x-\ep)_+\right\|=0.$$
We can therefore find a sequence $(t_m)_{m\in\N}$ in $M_k(A^\alpha)$ such that
\[\left\|t_m^*t_m-(x-\ep)_+\right\|<\frac{1}{m} \ \ \mbox{ and } \ \ \left\|t_mt_m^*-ss^*\right\|<\frac{1}{m}\]
for all $m\in\N$.
We deduce that
\[\left[\left(x-\ep-\frac{1}{m}\right)_+\right]\leq \left[t_m^*t_m\right] = \left[t_mt_m^*\right]\]
in $\Cu(A^\alpha)$. Taking limits as $m\to \I$, and using R\o rdam's Lemma again, we conclude that
\[[(x-\ep)_+]\leq [ss^*] \leq [y]\]
in $\Cu(A^\alpha)$. Since $\ep>0$ is arbitrary, this implies that $[x]\leq [y]$ in $\Cu(A^\alpha)$,
as desired. This finishes the proof.\end{proof}

\begin{rem}
In the context of the theorem above, one can show that
if $\Cu(A)$ is finitely generated as a Cuntz semigroup,
then $\Cu(\iota)(\Cu(A^\alpha))$ is a direct summand of $\Cu(A)$
(although the splitting is not natural). However, very few \ca s have finitely generated
Cuntz semigroups (this, in particular, implies that $\Cu(A)$ is countable), and hence we
do not prove this assertion here.
\end{rem}

\begin{cor}\label{cor:K-thyCuCP}
Let $A$ be a unital, separable \ca, let $G$ be a second-countable compact group, and let
$\alpha\colon G\to\Aut(A)$ be an action with the \Rp.
\be\item There is a canonical identification of $K_\ast(A\rtimes_\alpha G)$ with an order-subgroup
of $K_\ast(A)$;
\item There is a canonical identification of $\Cu(A\rtimes_\alpha G)$ with a sub-semigroup of
$\Cu(A)$.\ee\end{cor}
\begin{proof} Recall that two Morita equivalent separable \ca s have canonically isomorphic
$K$-groups and Cuntz semigroup. The result then follows from Proposition~2.4 in \cite{Gar_crossed_2014}
together with \autoref{thm:InjK-Thy} or \autoref{thm:InjCu}.\end{proof}

Recall that an ordered semigroup $S$ is said to be \emph{almost unperforated} if for every
$x,y\in S$, and $n\in\N$ such that $(n+1)x\leq ny$, then $x\leq y$. Since
almost unperforation passes to sub-semigroups (with the induced order), the following is immediate.

\begin{cor}\label{cor:AlmostUnperf}
Let $A$ be a unital, separable \ca, let $G$ be a second-countable compact group, and let
$\alpha\colon G\to\Aut(A)$ be an action with the \Rp. If $\Cu(A)$ is almost unperforated, then
so are $\Cu(A^\alpha)$ and $\Cu(A\rtimes_\alpha G)$. \end{cor}

Under the presence of the Rokhlin property, the ideal structure of crossed products and fixed
point algebras can be completely determined: ideals are induced by invariant ideals in the
original algebra; see Theorem~2.13 in~\cite{Gar_crossed_2014}.
Moreover, we deduce that strict comparison of positive elements is preserved under formation
of crossed products and passage to fixed point algebras.

\begin{cor}\label{cor:StrComp}
Let $A$ be a simple, unital, separable \ca, let $G$ be a second-countable compact group, and let
$\alpha\colon G\to\Aut(A)$ be an action with the \Rp. If $A$ has strict comparison
of positive elements, then so do $A^\alpha$ and $A\rtimes_\alpha G$. \end{cor}
\begin{proof} By Lemma 6.1 in \cite{TikTom_structure_2015}, strict comparison of positive elements for a simple
\ca\ is equivalent to almost unperforation of its Cuntz semigroup. Since $A^\alpha$ and $A\rtimes_\alpha G$
are simple by Corollary~2.14 in~\cite{Gar_crossed_2014}, the result follows from \autoref{cor:AlmostUnperf}.\end{proof}

\section{Equivariant semiprojectivity and duality}

In this section, we explore the following question: \emph{are Rokhlin actions of abelian groups automatically dual actions?} 
Surprisingly, the answer is ``yes'' for compact Lie groups of dimension at most one; see \autoref{thm:RokAreDual}. 
Equivariant semiprojectivity, as defined by Phillips in~\cite{Phi_equivariant_2011}, is our main technical tool which allows
us to replace almost equivariant almost homomorphisms by honest equivariant homomorphisms. The conditions on $G$ mentioned
above ensure that $(C(G),\texttt{Lt})$ is equivariantly semiprojective; see \autoref{thm:C(G)eqsj}. The techniques developed
here also allow us to characterize commutative systems with the Rokhlin property, as those free systems for which the 
associated principal $G$-bundle is trivial; see \autoref{thm:CommSysts}. 

The following is essentially Definition~1.1 in \cite{Phi_equivariant_2011}; see also \cite{PhiSorThi_semiprojectivity_2015}

\begin{df}\label{df:EqSj}
Let $G$ be a locally compact group, let $A$ be a \ca, and let $\alpha\colon G\to\Aut(A)$ be a
continuous action. Let $\mathcal{B}$ be a class of \ca s.
We say that the triple $(G,A,\alpha)$ is \emph{equivariantly semiprojective with respect to $\mathcal{B}$},
if the following holds: given an action $\beta\colon G\to\Aut(B)$ of $G$ on a \ca\ $B$ in $\mathcal{B}$, given
an increasing sequence $J_1\subseteq J_2\subseteq\cdots\subseteq B$ of $G$-invariant ideals such that
$B/J_n$ is in $\mathcal{B}$ for all $n\in\N$, and, with
$J=\overline{\bigcup\limits_{n\in\N}J_n}$, an equivariant homomorphism $\phi\colon A\to B/J$, there
exist $n\in\N$ and an equivariant homomorphism $\psi\colon A\to B/J_n$ such that, with $\pi_n\colon B/J_n\to B/J$
denoting the canonical quotient map, we have $\phi=\pi_n\circ\psi$. In other words, the following lifting problem
can be solved:
\beqa \xymatrix{ & B\ar[d] \\
& B/J_n\ar[d]^-{\pi_n}\\
A\ar[r]_-{\phi}\ar@{-->}[ur]^-{\psi}&B/J.}\eeqa
In the diagram above, full arrows are given, and $n\in\N$ and the dotted arrow are supposed to exist and make the lower
triangle commute.

Similarly, $(G,A,\alpha)$ is said to be \emph{unitally equivariantly semiprojective (with respect to $\mathcal{B}$)} if
given the lifting problem as above with $B$ and $\psi$ unital, there exist $n$ and a unital equivariant homomorphism 
$\psi$ as above.
\end{df}

We will mostly use the above definition in the case where $\mathcal{B}$ is the class of all $C^*$-algebras
(in which case we speak about equivariantly semiprojective actions).
However, in the proof of \autoref{thm:CommSysts}, it will be convenient to take $\mathcal{B}$ to be the class
of all commutative \ca s.

\begin{rem}\label{rem:GcptFixPtOnto}
Let $G$ be a compact group, let $A$ and $B$ be \ca s, and let $\alpha\colon G\to\Aut(A)$
and $\beta\colon G\to\Aut(B)$
be continuous actions. If $\varphi\colon A\to B$ is a surjective, equivariant homomorphism, then
$\varphi(A^\alpha)=B^\beta$. Indeed, it is clear that $\varphi(A^\alpha)\subseteq B^\beta$. For the reverse inclusion,
given $b\in B^\beta$, let $x\in A$ satisfy $\phi(x)=b$. Then $a=\int\limits_G\alpha_g(x)\ dg$ (using normalized
Haar measure on $G$) is fixed by $\alpha$ and satisfies $\phi(a)=b$.\end{rem}

\begin{lma}\label{lemma:GxH}
Let $G$ be a locally compact group, let $A$ be a \ca, and let $\alpha\colon G\to \Aut(A)$ be a contiuous action.
Let $H$ be a compact group, and denote by $\gamma\colon G\times H\to \Aut(A)$ the action given by
$\gamma_{(g,h)}(a)=\alpha_g(a)$ for all $(g,h)\in G\times H$ and for all $a\in A$. If $(G,A,\alpha)$
is (unitally) equivariantly semiprojective, then so is $(G\times H,A,\gamma)$.
\end{lma}
\begin{proof} 
We prove the non-unital version, since the unital one is identical.
Let $(G\times H,B,\beta)$ be a $G\times H$-algebra, let $(J_n)_{n\in\N}$ be an increasing sequence of $\beta$-invariant
ideals in $B$, and set $J=\overline{\bigcup\limits_{n\in\N}J_n}$. For $n\in\N$, denote by $\pi_n\colon B/J_n\to B/J$
the canonical quotient map. Let $\varphi\colon A\to B/J$ be an equivariant
homomorphism. For ease of notation, we identify $H$ with $\{1_G\}\times H$.
Since $H$ acts trivially on $A$, we have $\varphi(A)\subseteq (B/J)^{H}$. 
By \autoref{rem:GcptFixPtOnto}, and since $H$ is compact, we have
\[\pi_n\left(\left(B/J_n\right)^{H}\right)=\left(B/J\right)^{H}.\]

By averaging over $H$, similarly to what was done in the comments before this lemma, it is easy to show that
$\overline{\bigcup\limits_{n\in\N}J_n^H}=J^H$. Also, for $n\in\N$, it is clear that $J_n^H$ is an ideal in $B^H$,
and that there is a canonical identification
\[(B/J_n)^H\cong B^H/J_n^H,\]
under which $\pi_n\colon B/J_n\to B/J$ restricts to the quotient map
\[\pi_n^H\colon B^H/J_n^H\to B^H/J^H.\]
Denote by $\iota_n\colon B^H/J_n^H\to B/J$ the canonical inclusion as the $H$-fixed point algebra, and
likewise for $\iota\colon B^H/J^H\to B/J$.
We thus have an associated diagram
\beqa
\xymatrix{ & B\ar[d] & B^{H}\ar[l]\ar[d] \\
& B/J_n\ar[d]_-{\pi_n} & B^{H}/J^{H}_n\ar[l]_-{\iota_n}\ar[d]^-{\pi^H_n} \\
A\ar@/_1.5pc/[rr]_-{\varphi_0} \ar@/^3.5pc/@{-->}[urr]^<<<<<<<<<<{\psi_0}\ar[r]^-{\varphi}&
B/J & B^{H}/J^{H}\ar[l]_-{\iota}.}
\eeqa

Regard $A$ and $B$ as $G$-algebras. Since the range of $\varphi$ really is contained in
$(B/J)^{H}\cong B^H/J^H$, there is a $G$-equivariant homomorphism $\varphi_0\colon A\to B^H/J^H$
such that $\varphi=\iota\circ\varphi_0$.
Use equivariant semiprojectivity of $(G,A,\alpha)$
to find $n\in\N$ and a $G$-equivariant homomorphism $\psi_0\colon A\to B^{H}/J_n^{H}$ such that
$\varphi_0=\pi^H_n\circ\psi_0$.

Set $\psi=\iota_n\circ\psi_0\colon A\to B/J_n$. Then $\psi$ is $G\times H$-equivariant, and satisfies
$\varphi=\pi_n\circ\psi$. We conclude that $(G\times H,A,\gamma)$ is equivariantly semiprojective.
\end{proof}

In our next result, which is of independent interest, we characterize those
compact groups $G$ for which $(G,C(G),\texttt{Lt})$ is equivariantly
semiprojective. The application we have in mind is to pre-dual actions (\autoref{thm:RokAreDual}), so we
only prove the result for abelian groups.

\begin{thm} \label{thm:C(G)eqsj}
Let $G$ be a compact abelian group. Then the dynamical system
$(G,C(G),\texttt{Lt})$ is unitally equivariantly
semiprojective if and only if $G$ is a Lie group with $\dim(G)\leq 1$.
\end{thm}
\begin{proof} Suppose that $(G,C(G),\texttt{Lt})$ is equivariantly semiprojective.
Then $C(G)$ is semiprojectivey by Corollary~3.12 in
\cite{PhiSorThi_semiprojectivity_2015}. By Theorem~1.2 in~\cite{SorThi_characterization_2012}, $G$ must be
an ANR with $\dim(G)\leq 1$. Now, ANR's are locally contractible by Theorem~2 in~\cite{Fox_characterization_1942}.
Since the action of $G$ on itself
by translation is faithful and transitive, it follows from Corollary~3.7 in~\cite{HofKra_transitive_2015}
that $G$ is a Lie group.

Conversely, let $G$ be an abelian compact Lie group with $\dim(G)\leq 1$. By Theorem~5.2~(a) in~\cite{Sep_compact_2007},
when $G$ is zero-dimensional, then it is is isomorphic to a product of copies of finite cyclic groups; when $G$ is 
one-dimensional, then it is isomorphic to the product of the circle $\T$ and a zero-dimensional abelian compact Lie
group. We prove the result by induction.

\vspace{0.3cm}

\emph{Base case 1: $G=\Z_m$ for some $m\in\N$.}
For $j\in\Z_m$, let $p_j\in C(\Z_m)$ denote the $j$-th vector of the standard basis. We will use the following
descriptions of $C(\Z_m)$: it is the universal unital $C^*$-algebra
generated by $m$ projections adding up to 1; and it is the universal
\uca\ generated by a unitary of order $m$. (The unitary is $z=\sum\limits_{j\in\Z_m}e^{\frac{2\pi i j}{m}}p_j$.)
With this in mind, an equivariant unital homomorphism $\phi$ from $C(\Z_m)$ is determined by a unitary
$w=\phi(z)$ of order $m$, such that the automorphism corresponding to $j\in\Z_m$ sends $w$ to
$e^{-\frac{2\pi i j}{m}}w$.

Use semiprojectivity of the \ca\ $C(\Z_m)$ (in the unital category) to find $n\in\N$ and a unital homomorphism
$\psi_0\colon C(\Z_m)\to B/J_n$ such that $\pi_n\circ\psi_0=\varphi$:
\beqa \xymatrix{ & B\ar[d] \\
& B/J_n\ar[d]^-{\pi_n}\\
C(\Z_m) \ar[r]_-{\varphi}\ar[ur]^-{\psi_0}&B/J.}\eeqa

Fix $\ep>0$ such that
\[2^m m \left(\frac{4\ep+6\ep^2}{1+3\ep}\right)<1.\]
(In particular, $\ep<1$.) For $j\in\Z_m$, set $q_j=\psi_0(p_j)$. By increasing $n$, we may assume that
\[\max_{j,k\in\Z_m} \left\|\beta^{(n)}_j(q_k)-q_{j+k}\right\|<\frac{\ep}{m}.\]

Set $u=\sum\limits_{j\in\Z_m}e^{-\frac{2\pi i j}{m}}q_j$, which is a unitary in $B/J_n$. For $j\in\Z_m$, we have
\[\left\|\beta^{(n)}_j(u)-e^{\frac{2\pi i j}{m}}u\right\|
\leq \sum_{k\in\Z_m} \left\|\beta^{(n)}_j(q_k)-q_{j+k}\right\|<\ep.\]
Set
\[x=\frac{1}{m}\sum_{j\in\Z_m} e^{-\frac{2\pi i j}{m}}\beta^{(n)}_j(u).\]
Then
\[\|x-u\|\leq \frac{1}{m}\sum\limits_{j\in\Z_m} \left\|e^{-\frac{2\pi i j}{m}}\beta^{(n)}_j(u)-u\right\|<\ep.\]
Since $u$ is a unitary, $x$ is invertible and $\|x\|\leq (1+\ep)$.
Hence $v=x(x^*x)^{-\frac{1}{2}}$ is a unitary in $B/J_n$. Since $\|x^*x-1\|<3\ep$, we have
\begin{align*}
\|v-x\|&= \left\|x(x^*x)^{-\frac{1}{2}}-x\right\| \\
&\leq  (1+\ep)\left\|(x^*x)^{-\frac{1}{2}}-1\right\|\\
&\leq  (1+\ep)\left\|(x^*x)^{-1}-1\right\|\\
&= (1+\ep)\left\|(x^*x)^{-1}(x^*x-1)\right\|\\
&\leq  (1+\ep)\frac{1}{1+3\ep}3\ep.
\end{align*}

Moreover,
given $j\in\Z_m$, it is clear that $\beta^{(n)}_j(v)=e^{\frac{2\pi i j}{m}}v$, since $\beta^{(n)}_j(x)=e^{\frac{2\pi i j}{m}}x$. On the other hand,
\begin{align*}
\pi_n(x)&=\frac{1}{m}\sum_{j\in\Z_m} e^{-\frac{2\pi i j}{m}}\pi_n(\beta^{(n)}_j(u))\\
&= \frac{1}{m}\sum_{j\in\Z_m} e^{-\frac{2\pi i j}{m}}e^{\frac{2\pi i j}{m}}\pi_n(u)\\
&=\pi_n(u)=\varphi(z),
\end{align*}
and thus
\[\pi_n(v)=\pi_n(x)=\pi_n(u)=\varphi(z).\]

We use the identity $a^m-b^m=(a-b)(a^{m-1}+a^{m-2}b+\cdots+ab^{m-2}+b^{m-1})$ at the second step, to estimate
as follows:
\begin{align*} \left\|v^m-1\right\|&\leq \|v^m-x^m\|+\|x^m-u^m\|+\|u^m-1\| \\
&\leq 2^m m\|v-x\|+ 2^m m\|x-u\| \\
&<2^m m \left(\frac{3\ep+3\ep^2}{1+3\ep}+\ep\right)<1.
\end{align*}
Find a continuous function $f$ on the spectrum of $v$ such that $f(v)$ is a unitary of order $m$.
Since continuous functional calculus commutes with homomorphisms, we have
\[\pi_n(f(v))=f(\pi_n(v))=f(\varphi(z))=\varphi(z),\]
and
\[\beta^{(n)}_j(f(v))=f(\beta^{(n)}_j(v))=e^{\frac{2\pi i j}{m}}f(v)\]
for all $j\in\Z_m$. Hence $f(v)$ determines a unital homomorphism $\psi\colon C(\Z_m)\to B/J_n$
by $\psi(z)=f(v)$, which is moreover equivariant and satisfies $\pi_n\circ\psi=\varphi$.
Hence $(\Z_m,C(\Z_m), \texttt{Lt})$ is equivariantly semiprojective, as desired.

\vspace{0.3cm}

\emph{Base case 2: $G=\T$.} The argument is similar to the case $G=\Z_m$, but somewhat simpler.
Denote by $z$ the canonical unitary generating $C(\T)$.
Use semiprojectivity of the \ca\ $C(\T)$ (in the unital category) to find $n\in\N$ and a unitary
$u\in B/J_n$ such that $\pi_n(u)=\varphi(z)$. (This is equivalent to having a unital homomorphism
$\psi_0\colon C(\T)\to B/J_n$ satisfying $\pi_n\circ\psi_0=\varphi$.)
By increasing $n$, we may assume that
\[\max_{\zeta\in\T} \left\|\beta^{(n)}_\zt(u)-\zt u\right\|<1.\]

Let $\mu$ denote the normalized Haar measure on $\T$, and set $x=\int_\T \zt^{-1} \beta^{(n)}_\zt(u) \ d\mu(\zt)$.
Then $\|x-u\|<1$ and thus $x$ is invertible. Set $v=x(x^*x)^{-\frac{1}{2}}$, which
is a unitary in $B/J_n$ satisfying $\beta^{(n)}_\zt(v)=\zt v$ for all $\zt\in\T$. Finally, since $\pi_n(u)=\varphi(z)$, we have
\[\pi_n(x)=\int_\T \zt^{-1} \pi_n(\beta^{(n)}_\zt(u)) \ d\mu(\zt)=\int_\T  \zt^{-1}\zt\varphi(z) \ d\mu(\zt)=\varphi(z),\]
and thus $\pi_n(v)=\pi_n(x)=\varphi(z)$. It follows that the unitary $v$ determines a unital homomorphism
$\psi\colon C(\T)\to B/J_n$ satisfying $\pi_n\circ\psi=\varphi$, which is moreover equivariant. This finishes
the proof.

\vspace{0.3cm}

\emph{Inductive step: $G=\Z_m\times H$ for some $m\in\N$, and $H$ is a group for which $C(H)$ is equivariantly semiprojective.}
We want to show that the $(\Z_m\times H)$-\ca\ $C(\Z_m\times H)$ is equivariantly semiprojective. To this end,
let $\beta\colon \Z_m\times H\to \Aut(B)$ be an action, let $(J_n)_{n\in\N}$ be an increasing sequence of $\beta$-invariant ideals
in $B$, and set $J=\overline{\bigcup_{n\in\N} J_n}$. For $n\in\N$, we denote by $\pi_n\colon B\to B/J_n$ and $\pi\colon B\to B/J$ the canonical
quotient maps. Let $\varphi\colon C(\Z_m\times H)\to B/J$ be a unital, equivariant homomorphism.

By \autoref{lemma:GxH}, the $(\Z_m\times H)$-\ca\ $C(\Z_m)$ is equivariantly semiprojective. Find $r\in\N$ and a unital,
equivariant homomorphism $\theta\colon C(\Z_m)\to B/J_r$. By dismissing the ideals $J_0,\ldots, J_{r}$, we can assume that
$\theta$ is a unital equivariant homomorphism into $B$. For $j\in\Z_m$, let $\delta_j\in C(\Z_m)$ be the indicator function
of $\{j\}$, and set $p_j=\theta(\delta_j)\in B$. Define
$B^{(0)}=\sum_{j\in\Z_m}p_jBp_j$, which is a unital subalgebra of $B$. For $n\in\N$, set $J_n^{(0)}=J_n\cap B^{(0)}$
and $J^{(0)}=J\cap B^{(0)}=\overline{\bigcup_{n\in\N}J_n^{(0)}}$. Since the image of $\varphi$ restricted to the canonical copy of $C(H)$ in
$C(\Z_m\times H)$ commutes with the images $\pi(p_j)$ of $p_j$ in $B/J$, this restriction determines a unital, equivariant homomorphism
$\varphi_0\colon C(H)\to B^{(0)}/J^{(0)}$.
We obtain the following commutative diagram, where the non-labelled arrow are the natural ones:
\beqa
\xymatrix{ && B\ar@{->>}[d] & B^{(0)}\ar@{_{(}->}[l]\ar@{->>}[d]\\
&& B/J_n\ar@{->>}[d] & B^{(0)}/J^{(0)}_n\ar@{_{(}->}[l]\ar@{->>}[d] \\
C(H)\ar@{^{(}->}[r]\ar@/_1.5pc/[rrr]_-{\varphi_0}\ar@/^3.5pc/@{-->}[urrr]^-{\psi_0} & C(\Z_m\times H)  \ar[r]^-{\varphi}&
B/J & B^{(0)}/J^{(0)}\ar@{_{(}->}[l].}
\eeqa

Since, by assumption, $C(H)$ is equivariantly semiprojective as an $H$-\ca, another use of \autoref{lemma:GxH} shows that
$C(H)$ is also equivariantly semiprojective as a $(\Z_m\times H)$-\ca. Thus there exist $n\in\N$ and a unital, equivariant homomorphism
$\psi_0\colon C(H)\to B^{(0)}/J_n^{(0)}$. Now, define a linear map $\psi\colon C(\Z_m)\otimes C(H)\to B/J_n$
by $\psi(\delta_j\otimes a)=\pi_n(p_j)\psi_0(a)$ for all $j\in\Z_m$ and $a\in C(H)$. Since $\pi_n\circ\theta$ and $\psi_0$ have commuting
ranges, it follows that $\psi$ is a well-defined unital homomorphisms, and it is immediate to check that it is equivariant. The result follows.
\end{proof}

Recall that the translation action of $G$ on $C(G)$ has the Rokhlin property. With this in mind, the 
result above suggests the following question:

\begin{qst} Let $A$ be a semiprojective \uca, let $G$ be a compact group, and let 
$\alpha\colon G\to\Aut(A)$ be an action with the Rokhlin property. Is $(G,A,\alpha)$
equivariantly semiprojective?
\end{qst} 

We apply \autoref{thm:C(G)eqsj} to deduce that certain Rokhlin actions are always dual actions.
Our result is new even in the well-studied case of finite groups.

\begin{thm} \label{thm:RokAreDual}
Let $A$ be a \uca, let $G$ be an abelian compact Lie group with $\dim(G)\leq 1$, and let
$\alpha\colon G\to\Aut(A)$ be an action with the Rokhlin property. Then there exist
an action $\widecheck{\alpha}$ of $\widehat{G}$ on $A^\alpha$, and an isomorphism
\[\varphi\colon A^{\alpha}\rtimes_{\widecheck{\alpha}}\widehat{G}\to A\]
such that $\varphi\circ \widehat{\widecheck{\alpha}}_g=\alpha_g\circ\varphi$
for all $g\in G$. In other words, $\alpha$ is a dual action.
\end{thm}
\begin{proof}
We claim that there is a unital, equivariant homomorphism $C(G)\to A$. Once we
show this, the result will be an immediate consequence of Theorem~4 in~\cite{Lan_duality_1979}.
(One should check that the algebra provided
by Landstad's theorem is really $A^\alpha$, but this is a straighforward verification.)

Let $\varphi\colon C(G)\to A_{\I,\alpha}\cap A'$ be an equivariant
unital homomorphism as in \autoref{df:Rp}. Note that $c_0(\N,A)\cong \bigoplus\limits_{m\in\N}A$.
Now, $(G,C(G),\texttt{Lt})$ is equivariantly semiprojective by \autoref{thm:C(G)eqsj},
so there exist $n\in\N$ and a unital equivariant homomorphism $\psi\colon C(G)\to \ell^\I_\alpha(\N,A)/\bigoplus\limits_{m=1}^n A$ making the following
diagram commute:
\beqa\xymatrix{ & \ell^\I_\alpha(\N,A)\ar[d]\\
& \ell^\I_\alpha(\N,A)/\bigoplus_{m=1}^n A\ar[d]\\
C(G)\ar[r]_-{\varphi}\ar[ur]^-{\psi} & A_{\I,\alpha}.}\eeqa

Since $\ell^\I_\alpha(\N,A)/\bigoplus\limits_{m=1}^n A$ is again isomorphic to $\ell_\alpha^\I(\N, A)$, it follows
that there is a unital equivariant homomorphism $C(G)\to A$, as desired.
\end{proof}

\subsection{Commutative systems with the Rokhlin property.}
Our next goal is to describe those compact dynamical systems $G\curvearrowright X$ whose induced action of $G$
on $C(X)$ has the Rokhlin property.
Under mild assumptions on $G$, Rokhlin actions of $G$ on $X$ are precisely the free actions whose associated 
principal $G$-bundle $X\to X/G$ is trivial; see \autoref{thm:CommSysts}. This result and the methods used in its
proof have found applications in the context of the Borsuk-Ulam conjecture \cite{BauDabHaj_noncommutative_2015}. 
Since the case of non-Lie
groups is the most relevant one, we will prove \autoref{thm:CommSysts} in a generality that does not 
require $G$ to be a Lie group, unlike in \autoref{thm:RokAreDual}. Since equivariant semiprojectivity of
$C(G)$ (even in the commutative category) implies that $G$ is a Lie group, we have to consider the following 
weakening of this notion, which is nevertheless sufficient for our purposes:

\begin{df}\label{df:ewsj}
Let $G$ be a locally compact group, let $A$ be a \ca, and let $\alpha\colon G\to\Aut(A)$ be a
continuous action. Let $\mathcal{B}$ be a class of \ca s.
We say that the triple $(G,A,\alpha)$ is \emph{weakly equivariantly semiprojective with respect to $\mathcal{B}$},
if the following holds: given an action $\gamma\colon G\to\Aut(C)$ of $G$ on a \ca\ $C$ in $\mathcal{B}$, and given 
an equivariant homomorphism $\varphi\colon A\to C_\I$, there exists an equivariant homomorphism 
$\psi\colon A\to\ell^\I(\N,C)$ such that $\varphi=\kappa_C\circ\psi$.
In other words, the following lifting problem
can be solved:
\beqa \xymatrix{ & \ell^{\I}(\N,C)\ar[d]^{\kappa_C} \\
A\ar[r]_-{\varphi}\ar@{-->}[ur]^-{\psi}& C_\I.}\eeqa
In the diagram above, full arrows are given, and the dotted arrow is supposed to exist and make the
triangle commute.

The unital version is defined as in \autoref{df:EqSj}.
\end{df}

%Weak equivariant semiprojectivity is indeed weaker than equivariant semiprojectivity. To see this, one applies
%the definition of equivariant semiprojectivity (\autoref{df:EqSj}) to $B=\ell^\I(\N,C)$ with ideals
%given by $J_n=\oplus_{m=0}^n C$. In this case, $J=c_0(\N,C)$ and $B/J=C_\I$. Equivariant semiprojectivity 
%gives us a lift into $\ell^\I(\N,C)/\bigoplus\limits_{m=0}^nC$ for some $n\in\N$. Since $\ell^\I(\N,C)/\oplus_{m=0}^nC$
%is naturally equivariantly isomorphic to $\ell^\I(\N,C)$, this proves the claim.

\begin{rem}\label{rem:wesjAlternative}
Weak equivariant semiprojectivity can be rephrased in terms of a certain approximate lift. In 
the context of \autoref{df:ewsj}, instead of asking for the existence of $\psi$ we could ask for the 
following: given $\ep>0$, and a finite subset $F\subseteq C$, there exists an equivariant homomorphism
$\theta\colon A\to \ell^\I(\N,C)$ such that $\|(\kappa\circ \theta)(a)-\varphi(a)\|<\ep$ for all $a\in F$.
\end{rem}

This remark gives us a big family of examples. For example:

\begin{eg}\label{eg:C(G)eqsjCommut}
Denote by $\mathcal{C}$ the class of all commutative \ca s. 
Let $G$ be a compact Lie group, let $X$ be a smooth manifold, and let $G\curvearrowright X$ be
a faithful smooth action. Then $C(X)$ is unitally $G$-equivariantly semiprojective with respect to $\mathcal{C}$.
\end{eg}
\begin{proof}
By Theorem~8.8 in~\cite{Mur_ganr_1983}, $X$ is an equivariant ANR in the sense of the definition
before Proposition~4.1 in~\cite{Mur_ganr_1983}. (Notice the standing assumption in
\cite{Mur_ganr_1983} that all groups are compact Lie groups.) An easy comparison of the definitions shows
that a $G$-space $X$ is an equivariant ANR if and only if $C(X)$ is unitally equivariantly semiprojective with
respect to the class $\mathcal{C}$ of commutative \ca s, so the proof is finished. 
\end{proof}

In particular, if $G$ is a compact Lie group, then $C(G)$ is weakly $G$-equivariantly semiprojective 
with respect to the class of all commutative \ca s. Moreover, there are non-Lie groups for which this is
also true:

\begin{prop}\label{prop:CantorGpWESj}
Let $(G_n)_{n\in\N}$ be a sequence of nontrivial compact Lie groups, and set $G=\prod\limits_{n\in\N}G_n$.
\be\item $(G,C(G),\texttt{Lt})$ is unitally weakly equivariantly semiprojective (but not equivariantly semiprojective) 
with respect to the class $\mathcal{C}$ of all commutative \ca s.
\item Suppose that $G_n$ is finite and abelian for all $n\in\N$, and set $H=G\times\T$. 
Then $(G,C(G),\texttt{Lt})$ and $(H,C(H),\texttt{Lt})$ are unitally weakly equivariantly semiprojective (but not 
equivariantly semiprojective) with respect to the class $\mathcal{A}$ of all \ca s.
\ee
\end{prop}
\begin{proof}
(1). We verify the condition
in \autoref{rem:wesjAlternative}. Let $(G,A,\alpha)$ be a $G$-\ca, and let $\varphi\colon C(G)\to A_\I$ be an 
equivariant homomorphism. Let $\ep>0$ and let $F\subseteq C(G)$ be a
finite subset. 
For $n\in\N$, set $K_n=G_1\times\cdots\times G_n$. Note that the $G$-\ca\
$C(K_n)$ is equivariantly semiprojective with respect to $\mathcal{A}$ by \autoref{eg:C(G)eqsjCommut} and 
\autoref{lemma:GxH}. Denote by $\iota_n\colon C(K_n)\to C(G)$ and $\pi_n\colon C(G)\to C(K_n)$
the canonical equivariant homomorphisms. 

Find $m\in\N$ such that
$\|(\iota_m\circ\pi_m)(a)-a\|<\ep$ for all $a\in F$. Using $G$-equivariant semiprojectivity
of $C(K_m)$, find an equivariant homomorphism $\theta_0$
making the following diagram commute:
\beqa
\xymatrix{ & & \ell^\I(\N,A)\ar[dd]^{\kappa_A}\\
& & \\
C(K_m)\ar[uurr]^-{\theta_0}\ar[r]^-{\iota_m} & C(G)\ar@{-->}[uur]_-{\theta}\ar@/^1pc/[l]^{\pi_m}\ar[r]_{\varphi} & A_\I
}
\eeqa

The proof is finished by setting $\theta=\theta_0\circ\pi_m$.

(2). The argument for $C(G)$ is very similar, using \autoref{thm:C(G)eqsj} instead of \autoref{eg:C(G)eqsjCommut}. 
When proving that $C(H)$ is weakly equivariantly
semiprojective, one takes $K_n$ to be $\T\times G_1\times\cdots\times G_n$, and argues identically.
\end{proof}

It is not true that $C(G)$ is weakly equivariant
semiprojective for every compact group $G$. For example, one can show that if $G$ is zero-dimensional
and $C(G)$ is weakly equivariantly semiprojective, then $G$ must have torsion. It is very possible that
all zero-dimensional compact groups for which $C(G)$ is weakly equivariantly semiprojective must be
products of finite groups. We have not explored this direction any further. 

%\begin{pbm} 
%Characterize those compact groups $G$ for which $C(G)$ is weakly equivariantly
%semiprojective (with respect to all $G$-\ca s, or with respect to the commutative ones), similarly to 
%what was done in \autoref{thm:C(G)eqsj} and \autoref{eg:C(G)eqsjCommut}. Is it true that $C(G)$
%is always weakly equivariantly semiprojective with respect to commutative $G$-\ca s? If so, is the
%condition $\dim(G)\leq 1$ enough to guarantee weak equivariant semiprojectivity with respect to all
%$G$-\ca s?
%\end{pbm}

Before proving the main result of this subsection, we need an easy lemma. 
We are thankful to Tron Omland for suggesting its proof.
We refer the reader to Appendix~A in~\cite{EchKalQuiRae_categorical_2006}
for the definitions of coactions and their crossed products, as well as for some of their basic features.

\begin{lma}\label{lma:TrivialCoact}
Let $G$ be a compact group, let $A$ be a commutative \ca, and let
\[\delta\colon C(Y)\to M(A\otimes C^*(G))\]
be a coaction of $G$ on $A$. If $A\rtimes_\delta G$ is commutative, then $\delta$ is trivial, this is,
$\delta(a)= a\otimes 1$ for all $a\in A$. In this case, there is a canonical isomorphism
\[A\rtimes_\delta G\cong A\otimes C(G).\]
\end{lma}
\begin{proof}
Denote by $j_{A}\colon A\to M(A\rtimes_\delta G)$ and $j_{G}\colon C(G)\to M(A\rtimes_\delta G)$
the universal maps (see, for example, Definition~A.39 in~\cite{EchKalQuiRae_categorical_2006}), which satisfy
the covariance condition
\[\left((j_{A}\otimes\id_G)\circ\delta\right)(a)=\left[(j_G\otimes\id_G)(w_G)\right](j_{A}(a)\otimes 1)\left[(j_G\otimes\id_G)(w_G)^*\right]\]
for all $a\in A$; see Definition~A.32 in~\cite{EchKalQuiRae_categorical_2006}.
Since $j_{A}(a)$ is in the center of $A\otimes M(C^*(G))$, which is dense in $M(A\otimes C^*(G))$,
the above identity becomes
\[\left((j_{A}\otimes\id_G)\circ\delta\right)(a)=j_{A}(a)\otimes 1\]
for all $a\in A$. This is equivalent to $j_A\otimes\id_G (\delta(a)-a\otimes 1)=0$ for all $a\in A$.

We claim that $\delta$ is normal (that is, that $j_A$ is injective; see Definition~2.1
in~\cite{Qui_full_1994}). Once we prove the claim, it will follow that
$\delta(a)=a\otimes 1$ for all $a\in A$, so $\delta$ is the trivial coaction.

We prove the claim. Since $G$ is amenable, $\delta$ is both a full and reduced coaction. Now, $\delta$
admits a faithful covariant representation by Proposition~3.2 in~\cite{Qui_full_1994}, so
it is normal by Lemma~2.2 in~\cite{Qui_full_1994}.

The last part of the statement follows immediately from the definition of the cocrossed product.
\end{proof}

What follows is a characterization of commutative dynamical systems with the Rokhlin property.
We are thankful to Hannes Thiel for providing the reference \cite{Mur_ganr_1983}.

We continue to denote by $\mathcal{C}$ the class of commutative \ca s.
Note that the assumption that $C(G)$ be weakly equivariantly semiprojective with respect to $\mathcal{C}$ is satisfied in many
cases of interest (see \autoref{eg:C(G)eqsjCommut} and \autoref{prop:CantorGpWESj}).

\begin{thm} \label{thm:CommSysts}
Let $X$ be a compact Hausdorff space, and let $G\curvearrowright X$ be an action of a compact
group $G$. Denote by $\alpha\colon G\to\Aut(C(X))$ the induced action. 
Assume that $C(G)$ is weakly equivariantly semiprojective with respect to $\mathcal{C}$.
Then $\alpha$ has the Rokhlin property if and only if there is a homeomorphism
\[\sigma\colon X/G \times G\to X\]
such that
\[g\cdot \sigma(Gx,h)= \sigma(Gx,gh)\]
for all $g,h\in G$ and for all $x\in X$.\end{thm}
\begin{proof}
The ``if" implication follows from part (1) of \autoref{thm: permanence properties}, since
the assumptions imply that there is an equivariant isomorphism
\[(G,C(X),\alpha)\cong (G,C(X/G)\otimes C(G),\id_{C(X/G)}\otimes \texttt{Lt}).\]

Let $\varphi\colon C(G)\to C(X)_{\I,\alpha}\cap C(X)'=C(X)_{\I,\alpha}$ be an equivariant
unital homomorphism as in \autoref{df:Rp}. By weak equivariant semiprojectivity, 
there exists a unital equivariant homomorphism
$\psi\colon C(G)\to \ell^\I_\alpha(\N,C(X))$ such that $\varphi=\kappa_{C(X)}\circ\psi$.
In particular, there is a unital equivariant homomorphism $C(G)\to C(X)$. Dually, there is an
equivariant, continuous map $\rho\colon X\to G$, which is necessarily surjective.

We show two ways of finishing the proof.

\emph{(First argument.)}
Denote by $\pi\colon X\to X/G$ the canonical quotient map onto the orbit space. (This map is a
principal $G$-bundle, but we do not need this here.) Define an equivariant continuous map
$\kappa\colon X\to X/G\times G$
by
\[\kappa(x)=(\pi(x),\rho(x))\]
for all $x\in X$. We claim that $\kappa$ is a homeomorphism.

To check surjectivity, let $(y,g)\in X/G\times G$ be given. Choose $x\in X$ such that $\pi(x)=y$,
and find $h\in G$ such that $\rho(h\cdot x)=g$. (Such element $h$ exists because the action of $G$
on itself is transitive.) It is then clear that $\kappa(h\cdot x)=(y,g)$, so $\kappa$ is surjective.

We now check injectivity. Let $x_1,x_2\in X$ satisfy $\kappa(x_1)=\kappa(x_2)$. Since $\pi(x_1)=\pi(x_2)$,
it follows that there exists $g\in G$ such that $g\cdot x_1=x_2$. Now, since $\rho(x_1)=\rho(x_2)$, we
must have $g=1_G$ and hence $x_1=x_2$.

It follows that $\kappa$ is a continuous bijection. Since $X$ and $X/G\times G$ are compact metric spaces,
it follows that $\kappa^{-1}$ is continuous, and the claim is proved. This finishes the proof.

\emph{(Second argument)}. Since there is a unital equivariant homomorphism $C(G)\to C(X)$, it follows from
Theorem~3
in~\cite{Lan_duality_1979}, that there are a coaction $\beta$ of $G$ on $C(X/G)$, and an isomorphism
\[C(X/G)\rtimes_\beta G\cong C(X)\]
that intertwines the dual $G$-action of $\beta$ and $\alpha$.
(The verification of the hypotheses of Theorem~3 in~\cite{Lan_duality_1979} takes slightly more work than for
Theorem~4 in~\cite{Lan_duality_1979}, which was needed in \autoref{thm:RokAreDual}, but it is nevertheless not
difficult. With the notation of Theorem~4 in~\cite{Lan_duality_1979}, observe that $\delta$ is will be nondegenerate
because $A$ is unital.)
Since the crossed product $C(X/G)\rtimes_\beta G$
is commutative, \autoref{lma:TrivialCoact} implies that the coaction $\beta$ must be trivial.
In this case, there is a canonical identification
$C(X/G)\rtimes_\beta G\cong C(X/G)\otimes C(G)$, which is moreover equivariant, with $\widehat{\beta}$ on the
left-hand side, and the action
$\id_{C(X/G)}\otimes \texttt{Lt}$ on the right-hand side.
The result follows.
\end{proof}

%\begin{rem}
%We note that in the proof of \autoref{thm:CommSysts}, we did not use the full
%strength of equivariant semiprojectivity (which can only happen for Lie groups);
%in fact, all we used is that $(G,C(G),\texttt{Lt})$ is
%equivariantly \emph{weakly} semiprojective (in the commutative category). There are many more
%compact groups for which $\texttt{Lt}$ is equivariantly semiprojectie. For example, totally
%disconnected groups of the form $\prod\limits_{n\in\N} \Z_{m_n}$ can be shown to have this property,
%and the conclusion of \autoref{thm:CommSysts} applies to these, with exactly the same proof.
%\end{rem}

It follows that if $G$ acts on $C(X)$ with the Rokhlin property, then the induced action
of $G$ on $X$ is free. The converse is not in general true: consider, for example the circle
on the M\"obius cylinder $M$ which is given by rotating each copy of the circle, and acting trivially
on the non-orientable direction. This action is free, and the orbit space is homeomorphic to $\T$.
However, $M$ is not homeomorphic to $\T\times\T$, and thus this action does not have the Rokhlin property.

On the other hand, we have the following partial converse:

\begin{prop}
Let a Lie group $G$ act freely on a compact Hausdorff space $X$. If
$\dim(G)=\dim(X)$, then the induced action of $G$ on $C(X)$ has the Rokhlin property.
\end{prop}
\begin{proof}
Since $G$ is a Lie group, we have $\dim(X/G)=\dim(X)-\dim(G)=0$.
Denote by $\pi\colon X\to X/G$ the canonical quotient map. By Theorem~8
in~\cite{Mos_sections_1956}, there exists a continuous cross-section $s\colon X/G\to X$.
Given $x\in X$, there exists $g_x\in G$ such that $g_x\cdot x=(s\circ \pi)(x)$.
Moreover, since the action is free, $g_x$ is uniquely determined by $x$ and $s$, and it
is easy to verify that the assignment $x\mapsto g_x$ is continuous, using continuity of
$s$ and of the group operations.

One readily checks that the map $\kappa\colon X\to X/G\times G$, given by
$\kappa(x)=(Gx,g_x)$ for all $x\in X$,
is an equivariant homeomorphism. We conclude that $G\to\Aut(C(X))$ has the Rokhlin property.
\end{proof}

%\bibliographystyle{siam}
%\bibliography{Ebibliography}

\end{document}